\documentclass[10pt]{article}
\usepackage{amsmath,amsthm,amsfonts,amssymb,epsfig,enumerate}
\usepackage{lineno,hyperref}
\usepackage{flafter}

\newtheorem{them}{Theorem}[section]

\newtheorem{lem}[them]{Lemma}

\newtheorem{rem}[them]{Remark}
\numberwithin{equation}{section}

\newcommand{\norm}[1]{\left\Vert#1\right\Vert}
\newcommand{\abs}[1]{\left\vert#1\right\vert}

\usepackage{graphicx}
\usepackage{subfigure}
\usepackage{float}
\usepackage{epstopdf}
\usepackage{color}
\usepackage{colortbl,xcolor}
\usepackage{booktabs}

\begin{document}
\title{On the simultaneous recovery of boundary impedance and internal conductivity}
\author{Jinchao Pan$^{\dag ,\ddag}$ \quad Jijun Liu$^{\dag, \ddag,}$ \thanks{Corresponding author: Prof. Dr. J.J.Liu, email: jjliu@seu.edu.cn}\\
$^\dag$School of Mathematics/S.T.Yau Center of Southeast University, Southeast University\\
Nanjing, 210096, P.R.China\\
$^{\ddag}$Nanjing Center for Applied Mathematics\\
Nanjing, 211135, P.R.China
}

\maketitle
\begin{abstract}
Consider an inverse problem of the simultaneous recovery of boundary impedance and internal conductivity in the electrical impedance tomography (EIT) model using  local internal measurement data, which is governed by a boundary value problem for an elliptic equation in divergence form with Robin boundary condition. We firstly express the solution to the forward problem by volume and surface potentials in terms of the Levi function. Then, for the inverse problem, we prove the uniqueness of the solution in an admissible set by unique extension of
the solution under some {a-prior} assumption. Finally we establish the regularizing reconstruction schemes for boundary impedance and internal conductivity using noisy measurement data with rigorous error estimates. The mollification method is proposed to recover the boundary impedance from the boundary condition, and the internal conductivity with known boundary value is recovered from an integral system, where the Tikhonov regularization is applied to seek the stable solution, considering that the error involved in the boundary impedance coefficient reconstruction will propagate to the recovering process for internal conductivity. Numerical implementations are presented to illustrate the validity of the proposed method.
\end{abstract}

\vskip 0.3cm
%

\vskip 0.3cm


\section{Introduction}
Consider the elliptic equation system
\begin{eqnarray}\label{liu11-01}
\begin{cases}
 -\nabla\cdot(\sigma(x)\nabla u(x))=0, & x \in \Omega,\\
\sigma(x)\partial_\nu u +h(x)u(x)=g(x), & x \in \partial \Omega
\end{cases}
\end{eqnarray}
in a bounded smooth domain $\Omega \subset \mathbb{R}^n$ for $n=2,3$,
with $0 < \sigma_0<\sigma(x)\in C^1(\bar\Omega)$ and $0<h_0\le h(x)\in C(\partial\Omega)$, where $\nu(x)$ is the outward normal direction on $\partial \Omega$, and $g(x)$ is the boundary source.  The distribution of electrical potential $u(x)$ in conductive media $\Omega$ with conductivity $\sigma(x)$ can be described by this model. The positive boundary impedance coefficient $h(x)$, which is also called the Robin coefficient, is of physical importance in engineering applications, e.g., describing the corrosion damage in electrostatic conductor \cite{Kaup, Inglese}.

The well-posedness of \eqref{liu11-01} has been considered thoroughly, see \cite{Gilbert} for the classical solution  and  \cite{Daners, Amr} for generalized solution in some Sobolev spaces. When both $h(x)$ and $\sigma(x)$ are known, the solution $u(x)$ to forward problem \eqref{liu11-01} can be solved numerical by finite element method (FEM) for the domain $\Omega$ in general shape. However, for many engineering situations such as the electrical impedance tomography models, the internal conductivity $\sigma(x)$ and boundary coefficient $h(x)$ may be unknown, which need to be detected using some extra measurable information.  That is, the inverse problem for $(\sigma, h)$ based on \eqref{liu11-01} can be considered as biomedical imaging model or nondestructive testing model in material sciences. More precisely,
for small $\epsilon>0$, denote by $\Omega_{\epsilon}:=\{x\in\Omega,\;\text{dist}(x,\partial \Omega)\leq\epsilon\}\subset\subset\Omega$ the internal domain of $\Omega$ near to the boundary $\partial\Omega$.
Assume we are given
\begin{eqnarray}\label{liu24-u_subdomian}
 u(x)=U(x),\quad x \in \Omega_{\epsilon},
\end{eqnarray}
or its noisy form $U^\delta$ satisfying
\begin{eqnarray}\label{liu24-u_error}
    \norm{U^{\delta}(\cdot)-U(\cdot)}_{H^1(\Omega_{\epsilon})}\le \delta.
\end{eqnarray}
We aim to recover $(\sigma(x)|_{\Omega},h(x)|_{\partial\Omega})$ approximately from
\eqref{liu11-01}-\eqref{liu24-u_error}.

Different from the classical EIT models using boundary Dirichlet-to-Neumann map to recover the internal conductivity, which essentially need the boundary measurements corresponding to infinite number of boundary current/potential injections, the above proposed model aims to detect both the internal conductivity $\sigma(x)$ and the boundary impedance $h(x)$, using the interior measurement corresponding to one boundary injection $g(x)$. Notice, our inversion input data are given only in a small internal layer $\Omega_{\epsilon}$ near to the domain boundary, which is implementable in many practical situations.

This nonlinear inverse problem belongs to the category of the determination of the unknown ingredients in elliptic system, which have been studied extensively for different inversion input, except for the well-known Calderon problems. Adesokan, Jensen and Jin considered the inversion of the conductivity in acousto-electric tomography \cite{AET_Adesokan}, by reformulating it as an optimization problem involved a date-fit term and TV penalty term. In \cite{Bal}, G. Bal proposed a numerical method for mixed conductivity imaging problem, and proved the convergence of the L-M iteration scheme with numerical simulations to validate the effect of iteration algorithm. Hubmer \cite{Hubmer} studied the conductivity reconstruction based on the acoustic and electrical conjugate imaging technology, where the internal power density function corresponding to the boundary current excitation with compact support is taken as inversion input, and this nonlinear inverse problem is solved by Landweber iteration process. Recently, B.Jin used neural network-based reconstruction technology \cite{neural_networks_Jin} to identify the conductivity image, which realizes the reconstruction by transforming the problem into a relaxed weighted minimum gradient problem and uses a fully connected feed-forward neural network to approximate its minimum value.  A recently work for recovering internal $\sigma$ in \eqref{liu11-01} with Dirichlet boundary condition can be found in \cite{Zhang-Liu}. As for recovering boundary impedance coefficients for various PDE models, some numerical methods involving regularizing schemes to get stable solutions have also been proposed,  such as mollification \cite{Wangyuchan, Zhangmengmeng},  regularized least-squares variational formulation \cite{Robin_Jin}, and  the Bayesian approximation error approach \cite{Nicholson}.

To the best of our knowledge, the research works on multiple parameters identifications for PDE systems are still rare. Ian Knowles  focused on the uniqueness of the inverse problems \cite{Knowles}, especially for the identifications of multiple coefficients. In \cite{Wangyuchan}, the simultaneous reconstruction of boundary impedance coefficient and internal source for diffusion system is considered. In principle, the recovery of multiple parameters  can always be treated under the framework of optimization scheme by minimizing some cost functional with penalty term. Even if there is no uniqueness for the multiple parameters reconstruction problem, such an optimization scheme can always yield some generalized solution. However, it is well-known that the optimization schemes for nonlinear inverse problems always suffer from  multiple local minimizers together with the large amount of computations. In order to overcome these disadvantages for our inverse problem, we consider the reconstruction scheme based on the integral system instead of the optimization version, from which the solvability of the inverse problem is rigorously established. More precisely, we firstly recover the Robin coefficient $h(x)$ based on the boundary condition using $U^\delta|_{\Omega_\epsilon}$ directly with unknown $\sigma|_\Omega$, and then reconstruct the internal conductivity $\sigma(x)$ through an integral equation system based on the Levi function scheme, with the reconstructed boundary impedance coefficient involved. Notice, for this proposed scheme where two unknowns are recovered successively, the quantitative influence of the reconstruction error for boundary impedance coefficient on the recovery of $\sigma$ should be carefully analyzed.

The boundary integral methods have shown the efficiency for solving forward and inverse problems of PDEs with constant coefficients such as Laplace equations and Helmholtz equations, where the fundamental solutions can be expressed explicitly. However, for PDE systems with variable coefficients, the concept of fundamental solution should be generalized to the Levi function from which the solution can also be expressed in integral forms with some density pair, see \cite{A. Pomp, Beshley} for a systematic study. Different from the PDEs with constant coefficients, the integral expression for solution to PDE with variable coefficients should involve volume potential \cite{Pomp, Mir}. For the corresponding inverse problems, the integral expression of the solution is of special importance theoretically and numerically, since when the inverse problems are solved by some iterative scheme for minimizing the cost functional, the adjoint operator applied at each iteration step can be easily derived from these integral expressions \cite{Brebbia, George, Colton1}.

Although the Levi function schemes have been successfully applied in solving various forward problems \cite{Cle2, Mikhailov2, JCPAN}, there are few research on the inverse problems for recovering the unknown coefficients of PDEs using these schemes. A recent exploration  is  \cite{Zhang-Liu}, where the authors recovered the medium conductivity based on the Levi function for elliptic equation model with Dirichlet boundary condition. However, only one unknown coefficient $\sigma(x)$ is recovered there without error analysis. Moreover, the reconstruction is based on the analytic extension of solution $u$, which requires the strong analyticity assumption on the unknown $\sigma$.

In this paper, we aim to recover both the unknown conductivity and the boundary Robin coefficient, where the inversion data $U$ are firstly extended to the whole domain $\Omega$ in terms of the Levi function scheme instead of the analytic extension, with the advantage  that the {\it a-prior} assumption on the unknown $\sigma$ is weakened. Then this nonlinear problem is converted into a linear system with respect to the density pair $(\mu,\psi)$ for Levi expression with the coefficient matrix of the linear system as inversion input, and then solve  $\sigma(x)$, provided that $\sigma(x)|_{\partial\Omega}$ be known. Such an observation reveals the new idea of our proposed scheme dealing with the nonlinearity of the inverse problem, namely, $\sigma$ is solved from a linear system with inversion input involved in the coefficient matrix, which is still nonlinear with respect to the inversion input data. To numerically evaluate the domain integral in Levi expression involving the density function $\mu$ in $\mathbb{R}^3$, we apply a domain
discretization-free method which  converts the volume integral into a boundary integral using the dual reciprocity method (DRM)  \cite{DRM_Nardini, DRM_book} and the radial integration method (RIM) scheme \cite{RIM_Gao, Jaw}. On the other hand, since we firstly recover the Robin boundary coefficient, the  reconstruction error in this step will contaminate unavoidably the performance of reconstruction for the conductivity $\sigma(x)$ in the second step. We establish a quantitative analysis on the error propagation for recovering $\sigma(x)$.

The paper is organized as follows. In section 2, we introduce an integral equation system with respect to the density pair $(\mu,\psi)$ for the elliptic equation with Robin boundary condition based on the Levi function. The solvability of this derived coupled system with respect to $(\mu,\psi)$ is rigorously proven. Then the uniqueness for recovering two unknown coefficients is proven. In section 3, the regularizing scheme for reconstructing $(h(x),\sigma(x))$
from the system \eqref{liu11-01}-\eqref{liu24-u_error} is established, with  rigorous analysis on the choice strategies for the regularization parameters  and error estimates for recovering the unknowns. In section 4, the dual reciprocity method (DRM) and the radial integration method (RIM) which transforms the volume integrals into surface integrals are employed to evaluate the volume potential. In section 5, we implement numerical experiments of the proposed schemes for 2-dimensional and 3-dimensional cases, supporting our analysis and showing the validity of our proposed scheme. Section 6 is an appendix where we give an improved estimate on the regularizing operator, which is required for our error analysis.

\section{Levi function and uniqueness for inverse problem}
For differential operator $\mathcal{L}\diamond:=-\nabla\cdot(\sigma\nabla\diamond)$ in $\mathbb{R}^n$, we call the function $P(x,y)$ with $x,y\in \mathbb{R}^n$ the Levi function, provided
\begin{equation}\label{liu22-01}
\mathcal{L}_x P(x,y)=\delta(x-y)+R(x,y)
\end{equation}
in the distribution sense,
where $\delta$ is the standard Dirac function, and $R(x,y)$ is a function of weak singularity for $x=y$.

For $\Phi(x,y)$ being the fundamental solution to the Laplacian $-\Delta$, it is easy to verify that
$$P(x,y):=\frac{1}{\sigma(y)}\Phi(x,y),\quad R(x,y)=\frac{-1}{\sigma(y)}\nabla_x\Phi(x,y)\cdot \nabla\sigma(x)=\begin{cases}
 \frac{1}{\sigma(y)}\frac{(x-y)\cdot \nabla\sigma(x)}{2\pi|x-y|^2}, &n=2\\
 \frac{1}{\sigma(y)} \frac{(x-y)\cdot\nabla\sigma(x)}{4\pi|x-y|^3}, &n=3
\end{cases}$$
for $x\not=y$ meets \eqref{liu22-01}. That is, $P(x,y)$ is the Levi function.

Then the Levi function scheme  represents the solution to the forward problem (\ref{liu11-01}) as
\begin{equation}\label{liu24-solution}
u(x):=\int_{\Omega}\mu(y)P(x,y)dy+\int_{\partial\Omega}\psi(y)P(x,y)ds(y), \quad x\in\Omega.
\end{equation}
By the $H^2$-regularity of volume potential (Theorem 8.2, \cite{Colton1}) and
the jump relation of the normal derivative of single-layer potential as well as continuity
of single-layer potential (Theorem 2.19, Theorem 2.12 \cite{Colton}),
if $u(x)$ given by (\ref{liu24-solution}) solves (\ref{liu11-01}),
$(\mu,\psi)\in C(\Omega)\times C(\partial\Omega)$ must solve
\begin{eqnarray}\label{liu24-system0}
\begin{cases}
\mu(x)+\int_{\Omega}\mu(y) R(x,y)dy+\int_{\partial\Omega}\psi(y) R(x,y))ds(y)=0, &x\in\Omega\\
\int_{\Omega}\mu(y)E_{h}(x,y)dy+\int_{\partial\Omega}\psi(y)E_{h}(x,y)ds(y)+
\psi(x)
=2g(x), &x\in\partial\Omega,
\end{cases}
\end{eqnarray}
where $E_{h}(x,y)=2\sigma(x)\partial_{\nu(x)}P(x,y)+2h(x)P(x,y)$
involving $(\sigma|_{\partial\Omega}, h|_{\partial\Omega})$. This linear system can be written in the operator form
\begin{eqnarray}
    \left(
    \begin{array}{cc}
    I+R_{\Omega} &R_{\partial\Omega}\\
    E_{\Omega} &I+E_{\partial\Omega}
    \end{array}
    \right)
    \left(
    \begin{array}{c}
         \mu \\
         \psi
    \end{array}
    \right)=
    \left(
    \begin{array}{c}
         0\\
         2g
    \end{array}
    \right),
\end{eqnarray}
with the operators $R_\Omega, R_{\partial\Omega},E_\Omega, E_{\partial\Omega}$ defined from  \eqref{liu24-system0} explicitly.

The following result ensures that the Levi function scheme is applicable for solving our forward problem, which is crucial to  our reconstruction scheme for recovering $\sigma(x)$.

\begin{them}\label{them1}
Assume $0<(\sigma,h)\in C^1(\overline\Omega)\times C(\partial\Omega)$. Then it follows
\begin{itemize}

\item For any  $g\in  L^2(\partial\Omega)$,
there exists a unique solution $u\in H^2(\Omega)$ to the direct problem \eqref{liu11-01}.

\item For $g\in C(\partial\Omega)$, there exists a unique solution
$(\mu,\psi)\in C(\Omega)\times C(\partial\Omega)$ to \eqref{liu24-system0}. Moreover,
\begin{equation}\label{liu24-repre}
u(x):=\int_\Omega\mu(y)P(x,y)dy+\int_{\partial\Omega}\psi(y)P(x,y)ds(y), \quad x\in \mathbb{R}^n
\end{equation}
solves the direct problem \eqref{liu11-01} in $H^2(\Omega)$.

\end{itemize}


\end{them}

\begin{proof}

The unique existence of the solution $u\in W^{1,p}(\Omega)$ for $p\in (1,+\infty)$ to elliptic equation in general form with Robin boundary condition has been established in \cite{Amr}.
Now we show that the solution $u\in W^{1,p}$, which meets \eqref{liu11-01},
is also in $H^2(\Omega)$ in our situation. Since
$$0=\int_\Omega\nabla\cdot(\sigma\nabla u) u dx\le -\int_{\partial\Omega}h(x)u^2(x)ds(x)+\int_{\partial\Omega}g(x)u(x)dx$$
due to $\sigma>0$, we have
$\|u\|_{L^2(\partial\Omega)}\le \frac{1}{h_-}\|g\|_{L^2(\partial\Omega)}$ with $h_-:=\min_{\partial\Omega}h(x)>0$.
So the regularity of the solution to the elliptic equation with Neumann boundary condition (Theorem 2, \cite{Jer}) together with the above estimate means
$$\|u\|_{H^{3/2}(\Omega)}\le C(\Omega)\|g-hu\|_{L^2(\partial\Omega)}\le C(\Omega)(1+\frac{h_+}{h_-})\|g\|_{L^2(\partial\Omega)}$$
with $h_+=\max_{\partial\Omega} h(x)>0$.
Then for $0<\sigma\in C^1(\overline\Omega)$, it follows from the equation that
$$\|\Delta u\|_{L^2(\Omega)}\le \|\nabla \ln\sigma\cdot \nabla u\|_{L^2(\Omega)}\le
C\|u\|_{H^1(\Omega)}\le C\|u\|_{H^{3/2}(\Omega)}\le C\|g\|_{L^2(\partial\Omega)}.$$
That is, we have $\|u\|_{H^2(\Omega)}\le C\|g\|_{L^2(\partial\Omega)}$ and consequently $u\in H^2(\Omega)$.

The fact that $u$ defined by \eqref{liu24-repre} for $(\mu,\psi)$ satisfying \eqref{liu24-system0} solves \eqref{liu11-01} follows from the regularity and jump relation of volume and surface potentials. So it is enough to prove the solvability of $(\mu,\psi)\in C(\overline\Omega)\times C(\partial\Omega)$ to \eqref{liu24-system0}.
Since both $R(x,y)$ and $E_h(x,y)$ are of weak singularity,  $(R_\Omega, R_{\partial\Omega})$ is compact from $C(\overline \Omega)\times C(\partial\Omega)$ to $C(\Omega)$, and $(E_\Omega, E_{\partial\Omega})$ is compact from $C(\overline \Omega)\times C(\partial\Omega)$ to $C(\partial\Omega)$. So we only need to prove the uniqueness of the solution to \eqref{liu24-system0} by the Fredholm alternative theorem.

By the regularity of volume and surface potentials,  $u(x)$ given by \eqref{liu24-repre} with density pair $(\mu, \psi)\in C(\overline\Omega)\times C(\partial\Omega)$ is in $H^2(\Omega)$. Moreover, for $(\mu, \psi)$ satisfying \eqref{liu24-system0} with $g=0$,  $u$ solves \eqref{liu11-01} with $g=0$,
which leads to $u=0$ in $H^2(\Omega)$ and consequently $u=0$ in $C(\overline\Omega)$.

We already  have
\begin{align}\label{liu24-proof2}
u(x)=\int_{\Omega}\Phi(x,y)\frac{\mu(y)}{\sigma(y)}dy+\int_{\partial\Omega}\Phi(x,y)\frac{\psi(y)}{\sigma(y)}ds(y)\equiv 0
\end{align}
for $x\in\overline\Omega$.  However, $u(x)$ representing by \eqref{liu24-repre} also satisfies the exterior problem
\begin{eqnarray*}
\begin{cases}
-\Delta u(x)=0, & x\in\mathbb{R}^n\setminus \overline\Omega,\\
\sigma(x)\frac{\partial u}{\partial \nu} +h(x)u(x)=0, & x\in \partial \Omega
\end{cases}
\end{eqnarray*}
with zero assymptoic behavior $\abs{x}\to\infty.$ Consequently we also have
\begin{align}\label{liu24-proof3}
    u(x)=\int_{\Omega}\Phi(x,y)\frac{\mu(y)}{\sigma(y)}dy+\int_{\partial\Omega}\Phi(x,y)\frac{\psi(y)}{\sigma(y)}ds(y)\equiv 0
\end{align}
for $x\in \mathbb{R}^n\setminus\overline\Omega$. Taking $x\to\partial\Omega$ from $\Omega$ and $\mathbb{R}^n\setminus\overline\Omega$ in \eqref{liu24-proof2}
and \eqref{liu24-proof3} respectively, the continuity of single layer potential and the volume potential on $\partial \Omega$ together with the jump relation of derivative of single layer potential on $\partial\Omega$ yields
$$
\int_{\Omega}\Phi(x,y)\frac{\mu(y)}{\sigma(y)}dy+\int_{\partial\Omega}\Phi(x,y)\frac{\psi(y)}{\sigma(y)}ds(y)\pm\frac{\psi(x)}{\sigma(x)}\equiv 0,\quad x\in\partial\Omega,
$$
which generate $\psi(x)\equiv 0$ in $\partial\Omega$ by subtracting these two identities.
Then $\mu(x)\equiv0$ follows from
\begin{eqnarray*}
\mu(x)+\int_{\Omega}\mu(y)R(x,y)dy=0, \quad x\in\Omega.
\end{eqnarray*}
So, for $g\in C(\partial\Omega)$, \eqref{liu24-system0} has a unique solution $(\mu,\psi)\in C(\tilde\Omega)\times C(\partial\Omega)$.
The proof is complete.
\end{proof}


Now we consider the uniqueness of inverse problem.
Define the admissible set
$$
\mathbb{A}=\{(\sigma,h)\in C^1(\overline{\Omega})\times C(\partial\Omega):\;0<h_{-}\le h(x) \le h_{+}, \; 0<\sigma_-\le \sigma(x) \le \sigma_+\}.
$$

For given $g(x)\in C(\partial\Omega)$, denote by $u[\sigma,h](x)\in H^{2}(\Omega)\hookrightarrow C^{0,\lambda}(\overline\Omega)$ with $\lambda\in (0,1/2)$ the solution to forward problem \eqref{liu11-01}. The following  bounds on $u[\sigma,h](x)$ are necessary for estimating the convergence property of our regularizing solution for the inverse problem.

\begin{lem}\label{lem1}
Assume $\partial\Omega\in C^{1}$. If $g\in C(\partial\Omega)$ satisfies
$0\not\equiv g(x)\le 0$, then it follows for all $(h,\sigma)\in \mathbb{A}$ uniformly that
\begin{eqnarray}\label{liu24-bound_u}
    C_{-}(\Omega,\sigma_{\pm},h_{\pm},g)\le u[\sigma,h](x)\le C_{+}(\Omega, \sigma_{\pm}, h_{\pm},g)<0,\quad x \in \overline{\Omega}.
\end{eqnarray}
\end{lem}
\begin{proof}
By Theorem 2.1, it follows $u\in H^{2}(\Omega)$. Now we show $\max_{\overline \Omega }u:=u_0<0$.

Case 1. $u$ is constant in $\Omega$. Then the boundary condition becomes
$h(x)u_0=g(x)\le 0$ for $x\in\partial\Omega$
which leads to
$\frac{g(x)}{h(x)}\equiv u_0$ in $\partial\Omega$ due to $h(x)|_{\partial\Omega}>0$. Since $0\not\equiv g(x)\le 0$, we have $u_0<0$. That is, this situation occurs only for $\frac{g(x)}{h(x)}$ being a constant in $\partial\Omega$ together with $g(x)|_{\partial\Omega}<0$.

Case 2. $u$ is not constant in $\Omega$. If $\max_{\overline \Omega }u=u_0\ge 0$, then $u$ attains its maximum value $u_0$ only at some point $x_0\in\partial\Omega$ by strong maximum principle (Theorem 4, page 350 \cite{Evans}), which means $\frac{\partial u}{\partial \nu}|_{x_0}>0$ .  Therefore we are lead to  $$g(x_0)=\sigma(x_0)\partial_{\nu}u(x_0)+h(x_0)u(x_0)
\ge \sigma(x_0)\partial_{\nu}u(x_0)>0$$
 due to $\sigma(x_0), h(x_0)>0, u(x_0)\ge 0$, which contradicts $g(x_0)\le 0$.
So we have
 $u[\sigma,h](x)<0$ for $x \in \overline{\Omega}$.
On the other hand, the $C^{0,\lambda}$-estimate in Theorem 2.1 says
$|u(x)|\le \|u\|_{C^{0,\lambda }}\le C(\sigma,h,g)$,
which yields from $u<0$ that
$c(\sigma,h,g)<u(x)<C(\sigma,h,g)<0$ for $x\in\overline \Omega$.
The proof is complete.
\end{proof}

We have shown that the solution $u$ to \eqref{liu11-01} is in $H^2(\Omega)\hookrightarrow C^{0,\lambda}(\Omega)$ for $(\sigma, h)\in \mathbb{A}$ and $g\in C(\partial\Omega)$ in Theorem 2.1. Therefore our inversion input data $u[\sigma,h](x)=U(x)$ for $x\in\Omega_{\epsilon}$ is point wisely defined.
Now it is  natural to consider the uniqueness of our inverse problem:  can the inversion input  $u[\sigma,h](\cdot)|_{\Omega_{\epsilon}}$ determine $(\sigma|_{\overline\Omega}, h|_{\partial\Omega})$ uniquely?

We assume that $\sigma|_{\partial\Omega}$ is known. Since the normal derivative $\partial_{\nu}u[\sigma,h]|_{\partial\Omega}$ can be uniquely determined from $u[\sigma,h](\cdot)|_{\Omega_{\epsilon}}$, $h(x)|_{\partial\Omega}$ can be uniquely recovered from the boundary condition and Lemma 2.2 by the explicit expression
\begin{equation}\label{liu-240813-1}
h(x)=\frac{1}{u[\sigma,h](x)}\left(g(x)-\sigma(x)\partial_{\nu}u[\sigma,h](x)\right), \quad x\in\partial\Omega.
\end{equation}

If $u[\sigma,h](x)|_\Omega$ is known, i.e., the inversion input $u[\sigma, h](\cdot)$ is specified in the whole domain $\Omega$ instead of $\Omega_\epsilon\subset\subset \Omega$, the unique recovery of $\sigma(x)|_{\Omega}$ with known $\sigma|_{\partial\Omega}$ has been established in \cite{Richter} from the first order partial differential equation
\begin{equation}\label{liu17}
\nabla\sigma(x)\cdot\nabla u(x)+\sigma(x) \Delta u(x)=0,\quad x\in\Omega
\end{equation}
with respect to $\sigma$, using the expression of $\sigma$ obtained from the integrals along characteristic lines under the {\it a-prior} assumption
$\min_{\overline\Omega}|\nabla u(x)|\ge \kappa_1>0$.

In the case that $u[\sigma, h]$ is analytic in $\Omega$,
$u[\sigma_1, h]=u[\sigma_2, h]$ in $\Omega_\epsilon$ ensures $u[\sigma_1, h]=u[\sigma_2, h]$ in $\Omega$, so the uniqueness result established in \cite{Richter} is still applicable to yield $\sigma_1=\sigma_2$ in $\Omega$ for the same inversion input data
$u[\sigma_1, h]|_{\Omega_\epsilon}=u[\sigma_2, h]|_{\Omega_\epsilon}$.
To remove the analytical requirement on the solution $u[\sigma, h](x)$, a possible way is to apply the unique extension of the solution to elliptic equation  to get $u[\sigma, h]|_{\Omega}$ from $u[\sigma, h]|_{\Omega_\epsilon}$. However, since the governed equation \eqref{liu17}
for $u=u[\sigma,h]$ is $\sigma$-dependent with $\sigma$ to be determined,  the unique continuation based on PDEs cannot be applied directly to ensure $u[\sigma_1, h]|_{\Omega}=u[\sigma_2, h]|_{\Omega}$ from $u[\sigma_1, h]|_{\Omega_\epsilon}=u[\sigma_2, h]|_{\Omega_\epsilon}$, since $u[\sigma_1, h]$ and $u[\sigma_2, h]$ obey different elliptic equations in $\Omega$.

However, if we add some extra restrictions on $(\Omega, \Omega_\epsilon)$, we can still prove
$$u[\sigma_1, h]|_{\Omega}=u[\sigma_2, h]|_{\Omega}$$
from $u[\sigma_1, h]|_{\Omega_\epsilon}=u[\sigma_2, h]|_{\Omega_\epsilon}$ stated as follows, due to the solution expression by the Levi function.

Since $h(x)$ can be determined uniquely, we consider the unique recovery of $\sigma(x)$ with known $(h(x), u(x))|_{\partial\Omega}$. Then  \eqref{liu11-01} can be rewritten as a Neumann problem
\begin{eqnarray}\label{liu24-01}
\begin{cases}
-\nabla\cdot(\sigma(x)\nabla u(x))=0, & x \in \Omega,\\
\sigma(x)\partial_\nu u =g(x)-h(x)u(x), & x \in \partial \Omega
\end{cases}
\end{eqnarray}
with known $u|_{\partial\Omega}$.
Define $$\tilde\mu(x):=\frac{\mu(x)}{\sigma(x)}, \;x\in\overline\Omega,\quad \tilde\psi(x):=\frac{\psi(x)}{\sigma(x)},\; x\in\Omega,\quad  \tilde h(x)=\frac{h(x)}{\sigma(x)},\; x\in\partial\Omega.$$
Then the solution expression \eqref{liu24-solution} can be written as
\begin{eqnarray}\label{liu24-new solution}
u(x)=\int_{\Omega}\tilde\mu(y)\Phi(x,y)dy+\int_{\partial\Omega}\tilde\psi(y)\Phi(x,y)ds(y), \quad x\in\Omega.
\end{eqnarray}
Substituting \eqref{liu24-new solution} into \eqref{liu24-01}, we have
\begin{eqnarray}\label{liu24-system}
\begin{cases}
\tilde\mu(x)+\int_{\Omega}\tilde \mu(y)\tilde R(x,y)dy+\int_{\partial\Omega}\tilde \psi(y)\tilde R(x,y)ds(y)=0, \;x\in\Omega,\\
\int_{\Omega}\tilde \mu(y)\partial_{\nu(x)}\Phi(x,y)dy+\int_{\partial\Omega}\tilde \psi(y)\partial_{\nu(x)}\Phi(x,y)ds(y)+
\frac{1}{2}\tilde\psi(x)
=\frac{g(x)}{\sigma(x)}-\tilde h(x)u(x), \;x\in\partial\Omega,
\end{cases}
\end{eqnarray}
where $\tilde R(x,y)=-\nabla_x\Phi(x,y)\cdot \nabla \ln\sigma(x)$.
Set $\tilde g(x):=\frac{g(x)}{\sigma(x)}$ for $x\in\partial\Omega$. Then the second equation of \eqref{liu24-system} and \eqref{liu24-new solution} lead to
\begin{eqnarray}\label{liu24-system2}
\begin{cases}
\int_{\Omega}\tilde\mu(y)\Phi(x,y)dy+\int_{\partial\Omega}\tilde\psi(y)\Phi(x,y)ds(y)=U(x), \;x\in\Omega_{\epsilon}\\
\int_{\Omega}\tilde \mu(y)\partial_{\nu(x)}\Phi(x,y)dy+\int_{\partial\Omega}\tilde \psi(y)\partial_{\nu(x)}\Phi(x,y)ds(y)+
\frac{1}{2}\tilde \psi(x)
=\tilde g(x)-\tilde h(x)U(x), \;x\in\partial\Omega
\end{cases}
\end{eqnarray}
using inversion input $U|_{\Omega_\epsilon}$ and
known $(\tilde g-\tilde h U)|_{\partial\Omega}$, which is a $\sigma|_\Omega$-independent linear system with respect to the new density pair $(\tilde\mu,\tilde\psi)$.
Define the operator
\begin{align*}
\mathbb{K}_{11}^{\Omega\to\Omega_{\epsilon}}[\tilde\mu](x):
&=\int_{\Omega} \Phi(x,y)
\tilde\mu(y)dy,&x\in\Omega_{\epsilon},\\
\mathbb{K}_{12}^{\partial\Omega\to\Omega_{\epsilon}}[\tilde\psi](x):
&=\int_{\partial\Omega}
\Phi(x,y)\tilde\psi(y)ds(y),&x\in\Omega_{\epsilon},\\
\mathbb{K}_{21}^{\Omega\to\partial\Omega}[\tilde\mu](x):
&=2\int_{\Omega}\partial_{\nu(x)}\Phi(x,y)
\tilde\mu(y)dy,  &x\in\partial\Omega,\\
\mathbb{K}_{22}^{\partial\Omega\to\partial\Omega}[\tilde\psi](x):
&=2\int_{\partial\Omega}
\partial_{\nu(x)}\Phi(x,y)
\tilde\psi(y)ds(y),      &x\in\partial\Omega.
\end{align*}
Then the operator form of   (\ref{liu24-system2}) can
 be represented as
\begin{eqnarray}\label{liu24-system3}
\begin{cases}
\mathbb{K}_{11}^{\Omega\to\Omega_{\epsilon}}[\tilde\mu](z)+
\mathbb{K}_{12}^{\partial\Omega\to\Omega_{\epsilon}}[\tilde\psi](z)=U(z),  &z\in \Omega_{\epsilon},\\
\mathbb{K}_{21}^{\Omega\to\partial\Omega}[\tilde\mu](x)+(\mathbb{I}+
\mathbb{K}_{22}^{\partial\Omega\to\partial\Omega})[\tilde\psi](x)=2\tilde g(x)-2\tilde h(x)U(x),&x\in\partial\Omega.
\end{cases}
\end{eqnarray}

We will prove the uniqueness of recovery of $\sigma(x)$ using $u[\sigma,h]|_{\Omega_\epsilon}$ as inversion input.
Denote by $U_i(z):=u[\sigma_i,h](z)$ for $z\in\Omega_\epsilon$ the inversion input corresponding to $(\sigma_i, h)$ with $i=1,2$. Since we assume $\sigma_1(x)=\sigma_2(x)$ for $x\in\partial\Omega$, it is easy to see $\tilde h_1(x)=\tilde h_2(x):=\tilde h(x)$ and $\tilde g_1(x)=\tilde g_2(x):=\tilde g(x)$ for $x\in\partial\Omega$.

We firstly show the unique extension of $u[\sigma, h](x)$ from $\Omega_\epsilon$ to $\Omega$.

\begin{them} \label{0814-2}
 Assume that $0$ is not the eigenvalue of the matrix operator
\begin{eqnarray*}
\mathbb{K}:=\left(
\begin{array}{cc}
\mathbb{K}_{11}^{\Omega\to\Omega_{\epsilon}} &\mathbb{K}_{12}^{\partial\Omega\to\Omega_{\epsilon}}\\
\mathbb{K}_{21}^{\Omega\to\partial\Omega} &\mathbb{I}+\mathbb{K}_{22}^{\partial\Omega\to\partial\Omega}
\end{array}
\right).
\end{eqnarray*}
Then we have  $$u[\sigma_1, h](x)=u[\sigma_2, h](x), \quad x\in\Omega $$
from $u[\sigma_1, h]|_{\Omega_\epsilon}=u[\sigma_2, h]|_{\Omega_\epsilon}$ and the {\it a-prior} condition $\sigma_1|_{\partial\Omega}=\sigma_2|_{\partial\Omega}$.
\end{them}

\begin{proof}
Denote by $(\tilde\mu_i,\tilde\psi_i)$ the solution to \eqref{liu24-system3} corresponding to $U_i(z)=u[\sigma_i,h](z)$ and $\tilde g_i(x)$. Then the difference functions
$$S\tilde\mu:=\tilde\mu_1-\tilde\mu_2, \quad
S\tilde\psi:=\tilde\psi_1-\tilde\psi_2$$
meet the linear system
\begin{eqnarray}\label{0813-2}
\left(
\begin{array}{cc}
\mathbb{K}_{11}^{\Omega\to\Omega_{\epsilon}} &\mathbb{K}_{12}^{\partial\Omega\to\Omega_{\epsilon}}\\
\mathbb{K}_{21}^{\Omega\to\partial\Omega} &\mathbb{I}+\mathbb{K}_{22}^{\partial\Omega\to\partial\Omega}
\end{array}
\right)
\left(
\begin{array}{c}
     S\tilde\mu  \\
     S\tilde\psi
\end{array}
\right)=
\left(
\begin{array}{c}
     0  \\
     0
\end{array}
\right),
\end{eqnarray}
noticing that $U_1=U_2$ in $\overline\Omega_\epsilon$.
By the assumption on the operator $\mathbb{K}$, \eqref{0813-2} has only zero solution $S\tilde \mu=0,\; S\tilde\psi=0$. That is, we have
$\tilde \mu_1=\tilde\mu_2, \;\tilde\psi_1=\tilde\psi_2$,
which means that the inversion input data $u[\sigma,h]|_{\overline\Omega_\epsilon}$, together with the boundary condition and known $(h,\sigma)|_{\partial\Omega}$, can uniquely determine the density pair $(\tilde\mu,\tilde\psi)$. Finally, the solution expression \eqref{liu24-new solution} using the density pair $(\tilde\mu_i,\tilde\psi_i)$ leads to
$u[\sigma_1,h](x)=u[\sigma_2,h](x),x\in\Omega$.
The proof is complete.
\end{proof}

\begin{rem}
The assumption that $0$ is NOT the eigenvalue of $\mathbb{K}$ is crucial for this result, which in fact gives some restrictions on $(\Omega_\epsilon, \Omega, \sigma|_{\partial\Omega})$.
\end{rem}

Based on this extension relation, now we can prove the uniqueness for our inverse problem taking $u|_{\Omega_\epsilon}$ as inversion input, using the same arguments for proving Lemma 2 in \cite{Richter}. The Levi function expression in fact gives an implementable scheme for the unique continuation.

\begin{them}\label{Them2}
We assume that
\begin{itemize}
    \item The boundary excitation $g(x)$ is chosen so that $u[\sigma, h](x)\in C^2(\Omega)$ and
\begin{equation}\label{liu11-01-04}
\min_{\overline\Omega}|\nabla u[\sigma, h](x)|\ge \kappa_1>0, \quad \forall \; (\sigma, h)\in \mathbb{A};
\end{equation}
   \item $0$ is NOT the eigenvalue of matrix operator $\mathbb{K}_{\tilde h}$;
\end{itemize}
Then the inversion input $u[\sigma,h](x)|_{\Omega_\epsilon}$ uniquely determine
$(\sigma, h)$ for known $\sigma|_{\partial\Omega}$.
\end{them}

\begin{proof}
Denote by $u[\sigma_i, h_i](x)$ the solution to \eqref{liu11-01} corresponding to $(\sigma_i, h_i)\in\mathbb{A}$ for $i=1,2$. We shall prove $(\sigma_1, h_1)=(\sigma_2, h_2)$ provided $u[\sigma_1, h_1](x)=u[\sigma_2, h_2](x)$ for $x\in\Omega_\epsilon$.

Since $u[\sigma_1, h_1](x)=u[\sigma_2, h_2](x)$ for $x\in\overline\Omega_\epsilon$, we then have from   $\sigma_1|_{\partial\Omega}=\sigma_2|_{\partial\Omega}$ and the boundary condition that
    $$
    (h_1(x)-h_2(x))u[h_1,\sigma_1](x)=h_2(x)(u[h_1,\sigma_1](x)-u[h_2,\sigma_2](x))=0, \quad x\in\partial \Omega.
    $$
By Lemma \ref{lem1}, we know that $u[h_1,\sigma_1](x)<0$ for $x \in \partial \Omega$. So  $h_1(x)=h_2(x):=h_*(x)$ in $\partial\Omega$.

Next we will prove $\sigma_1(x)=\sigma_2(x)$ for $x\in \Omega$. Define
$$\Sigma(x):=\sigma_1(x)-\sigma_2(x),\quad w(x):=u[\sigma_1, h_*](x)-u[\sigma_2, h_*](x),\quad x\in\Omega.$$
Then $\Sigma(x)$ meets the system
\begin{eqnarray}\label{liu11-01-05}
    \begin{cases}
        \nabla\Sigma\cdot \nabla u[\sigma_1,h_*]+\Sigma \;\Delta u[\sigma_1,h_*]=-\nabla\sigma_2\cdot\nabla w-\sigma_2\Delta w, &x\in\Omega\\
        \Sigma=0, &x\in\partial\Omega.
    \end{cases}
\end{eqnarray}
Since $\min_{\overline\Omega}|\nabla u[\sigma_1, h_*](x)|\ge \kappa_1>0$ by \eqref{liu11-01-04}, it follows from Lemma 2 in \cite{Richter} that
\begin{equation}
\|\Sigma\|_\infty\le C(u[\sigma_1,h_*])\left\{\max_{\partial\Omega}|\Sigma(x)|+
\frac{[u[\sigma_1,h_*]]}{\kappa_1^2}\|\nabla\sigma_2\cdot\nabla w+\sigma_2\Delta w\|_\infty\right\},
\end{equation}
where $[u[\sigma_1,h_*]]:=\max_{\overline\Omega}u[\sigma_1,h_*](x)-
\min_{\overline\Omega}u[\sigma_1,h_*](x)$ is bounded from Lemma 2.3. So the boundary condition in \eqref{liu11-01-05} and  $0\equiv w(x)\in C^2(\overline\Omega)$ from Theorem 2.3 yields
$\Sigma(x)\equiv 0$ in $\Omega$.
The proof is complete.
\end{proof}

The condition \eqref{liu11-01-04} for the forward problem is crucial to the unique determination of $\sigma$. A sufficient condition to ensure this lower positive bound has been given for
$-\nabla\cdot (\sigma\nabla u)=f(x)$
with nonzero source $f(x)$ and homogeneous Dirichlet boundary condition for $u$ in \cite{Richter}. However, in our case of $f(x)\equiv 0$ and Robin boundary condition, we have to assume it. Some sufficient conditions on $(\sigma,h,g)$ to ensure \eqref{liu11-01-04} is still unclear.


\section{Regularizing scheme with error estimates}

To establish convergence rate of the regularizing solution, we assume that $(\sigma,h)\in \mathbb{A}$ and the specified input data $g(x)$ have the required regularity such that $u\in C^{2,\alpha_0}(\Omega)$ for $\alpha_0\in (0,1)$, see \cite{Gilbert} for the details for ensuring such a regularity. Then we have the pointwise estimates
\begin{eqnarray}
    |U(x)-U(y)|,|\nabla U(x)-\nabla U(y)|<C |x-y|,\quad |\Delta U(x)-\Delta U(y)|<C|x-y|^{\alpha_0},\; x,y\in \overline{\Omega}_{\epsilon}
\end{eqnarray}
for the exact inversion input data.

Since our inversion scheme for recovering $(\sigma,h)$ is decomposed into two steps: firstly recover $h(x)$ using boundary condition and then recover $\sigma(x)$ using the governed system, we need to construct the regularizing scheme for these two steps respectively.
Notice, this two-step scheme is closely related which means that the reconstruction error in the first step has essential influence on the recovery of $\sigma(x)$, since the Robin boundary condition is applied to determine the density pair for the Levi representation from which the solution $u$ and then the conductivity is computed.

As for the first step, we take the same scheme as that applying in \cite{Wangyuchan} where the final value $u(x,T)$ for $x\in\Omega$ are applied to recover the time-independent boundary impedance coefficient $\sigma(x)$ for a parabolic system, which is essentially to deal with the  ill-posedness of the normal derivative in Robin boundary condition. Since the normal derivative $\frac{\partial u}{\partial \nu}|_{\partial\Omega}$ depends only on the values of $u$ in the neibourhood of $\partial\Omega$, the regularizing scheme proposed there is still applicable for our case that $u$ is only specified in $\Omega_{\epsilon}\subset\subset\Omega$, rather than the whole domain $\Omega$.

More precise, for noisy data $U^\delta\approx U$ specified in $\Omega_{\epsilon}$, we firstly introduce the mollification operator $R_{\alpha}$ with some regularizing parameter $\alpha>0$ by
\begin{eqnarray*}
    R_{\alpha}[U^\delta](x):=\int_{\mathbb{R}^n}\rho_{\alpha}(|x-y|)\tilde U^\delta(y)dy=\int_{|x-y|\le\alpha}\rho_{\alpha}(|x-y|)\tilde U^\delta(y)dy, \quad x\in \Omega_\epsilon,
\end{eqnarray*}
where $\tilde U^\delta$ is the extension of
$U^\delta$ from $\Omega_\epsilon$ to $\mathbb{R}^n$, which satisfies $\tilde U^\delta=U^\delta$ in $\Omega_{\epsilon}$ and $\hbox{supp} \;\tilde U^\delta \subset V$ for some $V$ containing $\Omega_\epsilon$, while the kernel function
$$\rho_{\alpha}(t):=\frac{1}{\alpha^n}\rho(\frac{t}{\alpha}), \quad t>0$$
for the specified function $0\le \rho(\cdot) \in C_0^{\infty}(\mathbb R^1)$ satisfying supp $\rho \subset (0,1)$ and $\int_0^{\infty}\rho(t)tdt=\frac{1}{2\pi} $.

Upon the above scheme,  the regularizing solution for  approximating  $h(x)$ using $U^\delta(x)$ can be expressed as
\begin{equation}\label{Liu24-regular-scheme-h}
    h^{\alpha,\delta}(x)=\frac{g(x)-\sigma(x)\partial_{\nu}R_{\alpha}[U^{\delta}](x)}{U(x)},\quad x \in \partial \Omega
\end{equation}
due to $U|_{\partial\Omega}<0$.
Using the estimate \cite{Kress}
$$\norm{\partial_{\nu}R_{\alpha}[U^{\delta}]-\partial_{\nu}U}_{L^2(\partial \Omega)}\le \norm{\partial_{\nu}R_{\alpha}[U^{\delta}]-\partial_{\nu}U}_{L^2(\partial \Omega_\epsilon)} \le C\norm{R_{\alpha}[U^{\delta}]-U}_{H^2(\Omega_\epsilon)},$$
it follows from the definition \eqref{Liu24-regular-scheme-h} that
\begin{align}
\norm{h^{\alpha,\delta}-h}_{L^2(\partial\Omega)}
&\le\norm{h^{\alpha,\delta}-h^{\alpha,0}}_{L^2(\partial\Omega)}+\norm{h^{\alpha,0}-h}_{L^2(\partial\Omega)} \nonumber\\
&\le \norm{\frac{\sigma}{U}}_{L^2(\partial\Omega)}\left(\norm{\partial_{\nu}R_{\alpha}[U^{\delta}]-\partial_{\nu}R_{\alpha}[U]}_{L^2(\partial \Omega)}+\norm{\partial_{\nu}R_{\alpha}[U]-\partial_{\nu}U}_{L^2(\partial \Omega)}\right)\nonumber\\
&\le C\left(\norm{R_{\alpha}[U^{\delta}-U]}_{H^2(\Omega_\epsilon)}+\norm{R_{\alpha}[U]-U}_{H^2( \Omega_\epsilon)}\right)\nonumber\\
&\le C\left(\norm{R_{\alpha}}_{\mathcal{L}(H^1(\Omega_\epsilon),H^2(\Omega_\epsilon))}\delta+\norm{R_{\alpha}[U]-U}_{H^2(\Omega_\epsilon)}\right).
\end{align}
On the other hand, we also have the bounds  $$\norm{R_{\alpha}}_{\mathcal{L}(H^1(\Omega_\epsilon),H^2(\Omega_\epsilon))}\le C\frac{1}{\alpha^{n/2}}, \quad \norm{R_{\alpha}[U]-U}_{H^2(\Omega_\epsilon)}\le C\alpha^{\lambda_0},$$ see section 6 and
\cite{Wangyuchan} for details. So
we have established the following error estimate
\begin{align}\label{liu24-error-h}
    \norm{h^{\alpha(\delta),\delta}-h}_{L^2(\partial\Omega)}\le C \delta^{\lambda_0/(n/2+\lambda_0)},
\end{align}
where the optimal regularizing parameter is chosen in terms of the noise level by $\alpha(\delta)=O(\delta^{1/(n/2+\lambda_0)})$,
with $\lambda_0:=\min\{\alpha_0,1/2\}$.

For $h^{\alpha(\delta),\delta}(x)$ determined by (\ref{Liu24-regular-scheme-h}), now we can construct the regularizing solution for $\sigma(x)|_{\Omega}$ corresponding to the system \eqref{liu11-01}-\eqref{liu24-u_error}.
Since the well-posedness of the linear integral system \eqref{liu24-system0} with respect to the density pair have been proved by Theorem \ref{them1}, we can safely apply Levi function scheme to solve inverse problem for $\sigma(x)$.

So we can firstly solve $(\tilde\mu(z),\tilde\psi(x))$ from \eqref{liu24-system3} and then insert $(\tilde\mu(z),\tilde\psi(x))$ into the first equation of \eqref{liu24-system}, yielding a linear equation with respect to unknown variable $\sigma(x)|_{\Omega}$.
However, the first equation in \eqref{liu24-system3} is an integral equation of the first kind which  reveals essentially the ill-posedness  of the extension of $\Omega_\epsilon$ to the whole domain $\Omega$. To overcome this deficiency and yield the stable solution $(\tilde\mu(x),\tilde\psi(x))$,
we modify \eqref{liu24-system3} into its Tikhonov regularization version
\begin{eqnarray}\label{liu24-regular-system4}
(\beta\mathbb{I}+\mathbb{K}^*\mathbb{K})
\left[\begin{array}{c}
\tilde\mu\\
\tilde\psi
\end{array}\right]
\left(\begin{array}{c}
z\\x
\end{array}\right)=\mathbb{K}^*
\left(\begin{array}{c}
U(z)\\
2\tilde g(x)-2\tilde h(x)U(x)\end{array}\right), \quad (z,x)\in \Omega_{\epsilon}\times\partial\Omega
\end{eqnarray}
for regularizing parameter $\beta>0$, where the operator $\mathbb{K}$ is defined in Theorem 2.3.
Hence, a regularizing density function pair $(\mu^{\beta},\psi^{\beta})$ will be solved stably from \eqref{liu24-regular-system4} for specified $(U,\tilde g, \tilde h)$.

Inserting this density function pair $(\mu^{\beta},\psi^{\beta})$ into the first
equation of system \eqref{liu24-system}, a
linear integral equation can obtained with respect to $\sigma(x)$ as
\begin{eqnarray}\label{liu24-sigma-system5}
\sigma(x)\tilde\mu(x)-\left(\int_{\Omega}\tilde \mu(y)\nabla_x \Phi(x,y)dy+\int_{\partial\Omega}\tilde \psi(y)\nabla_x \Phi(x,y)ds(y)\right)\cdot\nabla \sigma(x)=0.
\end{eqnarray}
For $i=1,\cdots,M$, $\sigma(x)$ can be approximated by $\sigma(x)\approx \Sigma_{i=1}^{M}\alpha_i l_i(x)$ with some basis functions $l_i(x)$ and coefficients $\alpha_i$. Using the known numerical formulas (such as finite element method), this linear system can be solved numerically with known boundary data $\sigma(x)|_{\partial\Omega}$.

However, the inversion input $U$ in $\Omega_\epsilon$ is in fact given by its noisy form $U^\delta$ and the boundary impedance coefficient $h(x)$ involved in the right-hand side of \eqref{liu24-regular-system4} is the reconstructed $h^{\alpha(\delta),\delta}$ with error. So we need to estimate the deduced error coming from $U^\delta$ and $h^{\alpha(\delta),\delta}$.

For the regularizing solution $h^{\alpha(\delta),\delta}$ given by (\ref{Liu24-regular-scheme-h}) with  specified $\alpha= \alpha(\delta)=O(\delta^{1/(2+\gamma^*)})$,
the regularizing system \eqref{liu24-regular-system4}
for the noisy data $(U^\delta, h^{\alpha(\delta),\delta})$ becomes
\begin{eqnarray}\label{liu24-regular-system4-delta}
(\beta\mathbb{I}+\mathbb{K}^*\mathbb{K})
\left[\begin{array}{c}
\tilde\mu^{\alpha,\beta,\delta}\\
\tilde\psi^{\alpha,\beta,\delta}
\end{array}\right]
\left(\begin{array}{c}
z\\x
\end{array}\right)=\mathbb{K}^*
\left(\begin{array}{c}
U^{\delta}(z)\\
2\tilde g(x)-2\tilde h^{\alpha,\delta}(x)U^{\delta}(x)\end{array}\right), \quad (z,x)\in \Omega_{\epsilon}\times\partial\Omega
\end{eqnarray}
from which we construct the density function $(\tilde\mu^{\alpha,\beta,\delta},\tilde\psi^{\alpha,\beta,\delta})$ from (\ref{liu24-regular-system4-delta}).

Define $\varphi(z,x):=(\tilde\mu(z),\tilde\psi(x))^\top$ for $(z,x)\in \Omega\times\partial\Omega$.
To estimate $\|\varphi^{\alpha,\beta,\delta}
-\varphi\|$, we utilize
\begin{align}\label{Liu24-error-estimate}
  \norm{\varphi^{\alpha,\beta,\delta}-\varphi}_{L^2(\Omega\times\partial\Omega)}\le&
  \norm{\varphi^{\alpha,\beta,\delta}-\varphi^{\alpha,\beta,0}}_{L^2(\Omega\times\partial\Omega)}
  +\norm{\varphi^{\alpha,\beta,0}-\varphi}_{L^2(\Omega\times\partial\Omega)}.
\end{align}

Define $\mathcal{R}_{\beta}:=(\beta\mathbb{I}+\mathbb{K}^*\mathbb{K})^{-1}\mathbb{K}^*$. Then it is easy to see
\begin{align*}
\varphi^{\alpha,\beta,\delta}-\varphi^{\alpha,\beta,0}
 &=\mathcal{R}_{\beta}\left(\begin{array}{c}
U^{\delta}(z)-U(z)\\
2\tilde h^{\alpha,0}(x)U^{\delta}(x)-2\tilde h^{\alpha,\delta}(x)U(x)\end{array}\right)\\
&=\mathcal{R}_{\beta}\left(\begin{array}{c}
U^{\delta}(z)-U(z)\\
2\left(\tilde h^{\alpha,0}(x)-\tilde h^{\alpha,\delta}(x)\right)U(x)\end{array}\right)\\
&=\mathcal{R}_{\beta}\left(\begin{array}{c}
U^{\delta}(z)-U(z)\\
2\frac{U(x)}{\sigma(x)}\left(h^{\alpha,0}(x)-h^{\alpha,\delta}(x)\right)\end{array}\right)=\mathcal{R}_{\beta}\left(
\begin{array}{c}
U^{\delta}(z)-U(z)\\
2\partial_\nu R_\alpha[U-U^\delta](x)
\end{array}
\right),
\end{align*}
where we apply $U^\delta|_{\partial \Omega}=U|_{\partial \Omega}$ and the expression of $h^{\alpha,\delta}$.
So the first term  in the right-hand side of \eqref{Liu24-error-estimate} has the estimate
 \begin{align}
     \label{Liu24-error-estimate1}
   \norm{\varphi^{\alpha,\beta,\delta}-\varphi^{\alpha,\beta,0}}_{L^2(\Omega\times\partial\Omega)}
   &\le \frac{1}{2\sqrt{\beta}}\delta+\frac{1}{\sqrt{\beta}}\norm{\partial_\nu R_\alpha[U-U^\delta]}_{L^2(\partial\Omega)}\nonumber\\
   &\le \frac{1}{2\sqrt{\beta}}\delta+\frac{C}{\sqrt{\beta}}\norm{R_\alpha[U-U^\delta]}_{H^2(\Omega_\epsilon)}\nonumber\\
   &\le \frac{1}{2\sqrt{\beta}}\delta+\frac{C}{\sqrt{\beta}}\norm{R_\alpha}_{\mathcal{L}(H^1(\Omega_\epsilon), H^2(\Omega_\epsilon))}\delta\le \frac{1}{2\sqrt{\beta}}\delta+\frac{C}{\sqrt{\beta}}\frac{1}{\alpha^{n/2}}\delta
 \end{align}
from the standard estimate on Tikhonov regularization.
Introduce $$\eta^{\alpha,0}(z,x):=(U(z),2\tilde g(x)-2\tilde h^{\alpha,0}(x)U(x))^\top, \;
\eta(z,x):=(U(z),2\tilde g(x)-2\tilde h(x)U(x))^\top$$
for $(z,x)\in \Omega_\epsilon\times\partial\Omega$.
Then it follows from
\begin{eqnarray}
(\beta\mathbb{I}+\mathbb{K}^{*}\mathbb{K})\varphi^{\alpha,\beta,0}(z,x)
=\mathbb{K}^{*}\eta^{\alpha,0}(z,x)
\end{eqnarray}
and $\mathbb{K}\varphi(z,x)=\eta(z,x)$ that
 \begin{align}\label{Liu24-error-estimate2}
  \norm{\varphi^{\alpha,\beta,0}-\varphi}_{L^2(\Omega\times\partial\Omega)}
  &=\norm{\mathcal{R}_{\beta} (\eta^{\alpha,0}-\eta)+(\mathcal{R}_\beta\eta-\varphi)}_{L^2(\Omega\times\partial\Omega)}\nonumber\\
&\le \norm{\mathcal{R}_{\beta}}\norm{\eta^{\alpha,0}-\eta}_{L^2(\Omega_\epsilon\times\partial\Omega)}
  +\norm{\mathcal{R}_{\beta}\mathbb{K}\varphi-\varphi}_{L^2(\Omega\times\partial\Omega)}.
\end{align}
It is easy to see $\norm{\eta^{\alpha,0}-\eta}\le 2\norm{\frac{U}{\sigma}}_{L^\infty(\partial\Omega)}\norm{h-h^{\alpha,0}}$. So \eqref{Liu24-error-estimate2} becomes
\begin{align}\label{Liu24-error-estimate3}
\norm{\varphi^{\alpha,\beta,0}-\varphi}_{L^2(\Omega\times\partial\Omega)}
  &\le 2\norm{\frac{U}{\sigma}}_{L^\infty(\partial\Omega)}\frac{1}{2\sqrt{\beta}}\norm{h-h^{\alpha,0}}_{L^2(\partial\Omega)}+
  \norm{\mathcal{R}_{\beta}\mathbb{K}\varphi-\varphi}\nonumber\\
  &\le \norm{\frac{U}{\sigma}}_{L^\infty(\partial\Omega)}\frac{1}{\sqrt{\beta}}\alpha^{\lambda_0}+
  \norm{\mathcal{R}_{\beta}\mathbb{K}\varphi-\varphi}.
\end{align}
Finally, the error estimate \eqref{Liu24-error-estimate} can be bounded from \eqref{Liu24-error-estimate1} and \eqref{Liu24-error-estimate3} to yield
\begin{align}\label{Liu24-error-estimate-beta}
  \norm{\varphi^{\alpha,\beta,\delta}-\varphi}_{L^2(\Omega\times\partial\Omega)}
  &\le \frac{C(\Omega,U|_{\partial \Omega},\sigma|_{\partial \Omega})}{\sqrt{\beta}}\left(\delta+\frac{1}{\alpha^{n/2}}\delta+\alpha^{\lambda_0}\right)+
  \norm{\mathcal{R}_{\beta}\mathbb{K}\varphi-\varphi}\nonumber\\
  &\le \frac{C(\Omega,U|_{\partial \Omega},\sigma|_{\partial\Omega})}{\sqrt{\beta}}\delta^{\lambda_0/(n/2+\lambda_0)}+
  \norm{\mathcal{R}_{\beta}\mathbb{K}\varphi-\varphi}.
\end{align}
with the optimal parameter  $\alpha(\delta)=O(\delta^{1/(n/2+\lambda_0)})$.

Based on this estimate, we have proven

\begin{them}
If the regularizing parameters $\alpha, \beta\to 0$ as $\delta\to 0$ is chosen in the way
\begin{equation}
\alpha(\delta)=O(\delta^{1/(n/2+\lambda_0)}),\quad \frac{\delta^{\lambda_0/(n/2+\lambda_0)}}{\sqrt{\beta}}\to 0,
\end{equation}
then the density pair has the convergence
$\varphi^{\alpha,\beta,\delta}\to \varphi$
as $\beta,\delta\to 0$. Moreover, the error bound is given by
\eqref{Liu24-error-estimate-beta}.
\end{them}

\begin{rem}
This theorem only gives the convergence property without convergence rate, although we have the convergence rate for Robin coefficient with optimal $\alpha=\alpha(\delta)$. It is well-known that, without the {\it a-prior} assumption on the exact solution $\varphi$, the convergence rate of  $\norm{\mathcal{R}_{\beta}\mathbb{K}\varphi-\varphi}\to 0$ can be arbitrarily slow. Since $\varphi$ is the auxiliary function introduced for the Levi function scheme, it is very hard to give a reasonable assumption to yield  convergence rate.
\end{rem}

Now we will estimate the reconstruction error of $\sigma$ using noisy inversion input $U^\delta$ by  \eqref{liu24-sigma-system5} corresponding to the noisy density function $(\tilde\mu^{\alpha,\beta,\delta},\tilde\psi^{\alpha,\beta,\delta})$.
Denote by $u$ the exact solution for exact density pair $\varphi$ from exact inversion input $U$.
By \eqref{liu24-new solution}, although the reconstruction error of $\sigma^{\alpha,\beta,\delta}-\sigma$ can be estimated by equation \eqref{liu24-sigma-system5} directly in terms of the density pair error, we still derive the error $\sigma^{\alpha,\beta,\delta}-\sigma$ theoretically in terms of $u^{\alpha,\beta,\delta}-u$, with $(u^{\alpha,\beta,\delta}, u)$ meeting
\begin{eqnarray}\label{Liu24-hyperbolic equation}
\begin{cases}
\nabla \sigma^{\alpha,\beta,\delta}\cdot\nabla u^{\alpha,\beta,\delta}+\sigma^{\alpha,\beta,\delta} \Delta u^{\alpha,\beta,\delta}=0,\quad x\in\Omega,\\
\nabla \sigma\cdot\nabla u+\sigma \Delta u=0,\quad x\in\Omega,\\
\sigma^{\alpha,\beta,\delta}=\sigma=\sigma^*, \quad x\in\partial\Omega.
\end{cases}
\end{eqnarray}

That is, the reconstruction error estimate on $\sigma$ can be transformed into the sensitivity analysis on  $\sigma$ with respect to $u$ based on \eqref{Liu24-hyperbolic equation} with known $\sigma|_{\partial\Omega}$. The result in space $L^{\infty}$ has been derived in \cite{Richter}, however it is unsuitable for our inverse problem where the error is measured by $L^2$-norm. In the sequel, the sensitivity analysis of identifying $\sigma$ will be derived in $L^2$ space under some restrictions on $u$ referring to
the proof process in \cite{Richter}. To this end, we firstly need to estimate $u^{\alpha,\beta,\delta}-u$ in terms of $\varphi^{\alpha,\beta,\delta}-\varphi$.

\begin{lem}\label{lem2-1}
 For $u^{\alpha,\beta,\delta}$  corresponding to the density pair
 $\varphi^{\alpha,\beta,\delta}$ and $u$ corresponding to the density pair $\varphi$, which are given by (\ref{liu24-new solution}), respectively, it follows that
 $$\|u^{\alpha,\beta,\delta}-u\|_{H^2(\Omega)}\le
 C_0\|\varphi^{\alpha,\beta,\delta}-\varphi\|_{L^2}.$$
\end{lem}
\begin{proof}
By the solution expression (\ref{liu24-new solution}), we have for $x\in\Omega$ that
\begin{equation}\label{liu240707-1}
 u^{\alpha,\beta,\delta}(x)-u(x)=\int_{\Omega}(\tilde\mu^{\alpha,\beta,\delta}(y)-\tilde\mu(y))\Phi(x,y)dy+\int_{\partial\Omega}(\tilde\psi^{\alpha,\beta,\delta}(y)-\tilde\psi(y))\Phi(x,y)ds(y).
\end{equation}

For the first term in the right-hand side of (\ref{liu240707-1}), which is a volume potential with respect to the density $\tilde\mu^{\alpha,\beta,\delta}-\tilde\mu$, it follows from Theorem 8.2 in  \cite{Colton1} that
\begin{equation}\label{liu240707-2}
\norm{\int_{\Omega}(\tilde\mu^{\alpha,\beta,\delta}(y)-\tilde\mu(y))\Phi(\cdot,y)dy}_{H^2(\Omega)}\le c_1 \|\tilde\mu^{\alpha,\beta,\delta}-\tilde\mu\|_{L^2(\Omega)}.
\end{equation}

For the second term
$$w(x):=\int_{\partial\Omega}(\tilde\psi^{\alpha,\beta,\delta}(y)-\tilde\psi(y))\Phi(x,y)ds(y), \quad x\in \overline \Omega,$$
the regularity of the single-layer surface potential (Theorem 3.6, \cite{Colton1}) ensures $\|w\|_{H^1(\partial\Omega)}\le c\|\tilde\psi^{\alpha,\beta,\delta}-\tilde\psi\|_{L^2(\partial\Omega)}$.  However, the trace theorem says $\|w\|_{H^1(\Omega)}\le c_2\|w\|_{H^{1/2}(\partial\Omega)}$. So we have
$
\norm{w}_{H^1(\Omega)}\le c\norm{\tilde\psi^{\alpha,\beta,\delta}-\tilde\psi}_{L^2(\partial\Omega)}$,
which leads to
\begin{equation}\label{liu240707-3}
\norm{w}_{H^2(\Omega)}\le c_3\norm{\tilde\psi^{\alpha,\beta,\delta}-\tilde\psi}_{L^2(\partial\Omega)}
\end{equation}
due to $\Delta w(x)+k^2w(x)=0$ in $\Omega$. Combining (\ref{liu240707-1})-(\ref{liu240707-3}) together, we have
\begin{align}\label{liu240707-4}
\norm{u^{\alpha,\beta,\delta}-u}_{H^2(\Omega)}&\le
c_1 \|\tilde\mu^{\alpha,\beta,\delta}-\tilde\mu\|_{L^2(\Omega)}+c_3\norm{\tilde\psi^{\alpha,\beta,\delta}-\tilde\psi}_{L^2(\partial\Omega)}
\nonumber\\&\le
\max\{c_1,c_3\}\norm{\tilde \varphi^{\alpha,\beta,\delta}-\tilde \varphi}_{L^2(\Omega\times \partial\Omega)}
\le
C_0\norm{\varphi^{\alpha,\beta,\delta}-\varphi}_{L^2(\Omega\times \partial\Omega)}\quad
\end{align}
for known $\sigma|_{\partial\Omega}$. The proof is complete.
\end{proof}

The next result is to estimate $\sigma^{\alpha,\beta,\delta}-\sigma$  by $u^{\alpha,\beta,\delta}-u$. In terms of  (\ref{Liu24-hyperbolic equation}), the solution of the conductivity is essentially to solve $\sigma$ from the hyperbolic equation of the first order along the characteristic lines.
\begin{lem}\label{lem2}
For given $g(x)<0$, assume the solution to the direct problem satisfies
\begin{equation}
\inf_{\Omega}|\nabla u[\sigma,h,g]|\ge \kappa_1>0
\end{equation}
uniformly for $(h,\sigma)\in \mathbb{A}$. Then we have
\begin{align}\label{liu24-sensitivity-sigmaL2}
&\norm{\sigma^{\alpha,\beta,\delta}-\sigma}_{L^2(\Omega)}\nonumber\\
\le&
\;D\left\{   \norm{\nabla \sigma^{\alpha,\beta,\delta}}_{L^2(\Omega)}\norm{\nabla(u-u^{\alpha,\beta,\delta})}_{L^2(\Omega)}
+\norm{\sigma^{\alpha,\beta,\delta}}_{L^2(\Omega)}\norm{\Delta(u-u^{\alpha,\beta,\delta})}_{L^2(\Omega)}\right\},
\end{align}
where the constant
$$
D=D(\mathbb{A},\Omega,g,k_1)= \exp\left(\frac{\norm{\sigma}_{L^{\infty}}C_b}{k_1\sigma_{-}}\right)\frac{C_b^{1/2}}{k_1^{3/2}}\abs{\Omega}^{1/2}
,\quad C_b=C_{+}(g,h_{-})-C_{-}(\sigma,h_{-},g,f).$$
\end{lem}

\begin{proof}
Firstly, (\ref{Liu24-hyperbolic equation}) leads to
\begin{eqnarray}\label{liu11-01-06}
\begin{cases}
\nabla(\sigma^{\alpha,\beta,\delta}-\sigma)\cdot\nabla u+
(\sigma^{\alpha,\beta,\delta}-\sigma)\Delta u=-\nabla\cdot(\sigma^{\alpha,\beta,\delta}\nabla (u^{\alpha,\beta,\delta}-u)), \quad x\in \Omega,\\
\sigma^{\alpha,\beta,\delta}-\sigma=0, \quad x\in\partial\Omega.
\end{cases}
\end{eqnarray}

Now we take
$$\Sigma:=\sigma^{\alpha,\beta,\delta}-\sigma, \quad w:=u,\quad F:=\nabla\cdot (\sigma^{\alpha,\beta,\delta}\nabla (u^{\alpha,\beta,\delta}-u)).$$
Then the PDE in \eqref{liu11-01-06} becomes
$$\nabla\Sigma\cdot\nabla w+\Sigma \Delta w=-F(x), \quad x\in\Omega.$$
Using the same arguments in \cite{Richter}, it follows that, for an arbitrary point $P\in\Omega$, denote by $\varsigma $ the corresponding characteristic through
$P$ and the boundary point $Q \in \partial \Omega$ at which $\varsigma$ originates, we have
\begin{equation}\label{Liu24-hyperbolic-case1}
\Sigma(P)=\Sigma(Q)\exp\left\{ -\int_{-l}^{0}\frac{\Delta w}{\abs{\nabla w}}ds'\right\}+\int_{-l}^{0}\frac{-F}{\abs{\nabla w}}\exp\left\{ -\int_{s}^{0}\frac{\Delta w}{\abs{\nabla w}}ds'\right\}ds,
\end{equation}
where $l$ is the arclength of characteristic $\varsigma$ between $Q$ and $P$. Then
\begin{align}\label{Liu24-hyperbolic-case1-1}
\abs{\Sigma(P)}
&\le\abs{\Sigma(Q)}e^{ql}+\frac{1}{k_1}\int_{-l}^{0}\abs{F(x(s))}e^{qs}ds\nonumber\\
&\le \abs{\Sigma(Q)}e^{ql}+\frac{1}{\kappa_1}\left (\int_{-l}^{0}F^{2}(x(s))ds \right)^{1/2} \left (\int_{-l}^{0}e^{2qs}ds\right )^{1/2}\nonumber\\
&\le \abs{\Sigma(Q)}e^{ql}+\frac{1}{k_1}\norm{F}_{L^2(\Omega)} \left (\frac{e^{2ql}-1}{2q}\right )^{1/2},
\end{align}
where $q\equiv \sup\limits_{\Omega}\left(\frac{-\Delta w}{\abs{\nabla w}}\right)$.
Noting that
$$e^{ql}\le \max \left\{1,e^{ql_{\text{max}}}\right\},  \left (\frac{e^{ql}-1}{ql}\right )\le \max \left\{1,e^{ql}\right\}, l_{\text{max}}\le \frac{[w]}{\inf_{\Omega}\abs{\nabla w}}, [w]=\sup\limits_{\Omega}w-\inf\limits_{\Omega}w,$$
we have
\begin{equation}\label{Liu24-hyperbolic-case1-2}
\abs{\Sigma(P)}
\le C(w)\left(\abs{\Sigma(Q)}+\frac{l_{\text{max}}^{1/2}}{\kappa_1}\norm{f}_{L^2(\Omega)}\right)
\le C(w)\left(\abs{\Sigma(Q)}+\frac{[w]^{1/2}}{\kappa_1^{3/2}}\norm{F}_{L^2(\Omega)}\right),
\end{equation}
where
$C(w)=\max\left\{1,\exp\left(\frac{q[w]}{k_1}\right)\right\}$.
Applying the Minkowski inequality,
it leads to
\begin{align}\label{Liu24-hyperbolic-case1-3}
\norm{\Sigma}_{L^2(\Omega)}&\le C(w)\left( \norm{\Sigma}_{L^2(\Gamma_1)}+\frac{[w]^{1/2}}{\kappa_1^{3/2}}\norm{F}_{L^2(\Omega)}\abs{\Omega}^{1/2}\right)\nonumber\\&\le
C(w)\left( \norm{\Sigma}_{L^2(\partial\Omega)}+\frac{[w]^{1/2}}{\kappa_1^{3/2}}
\norm{F}_{L^2(\Omega)}\abs{\Omega}^{1/2}\right).
\end{align}
Then (\ref{Liu24-hyperbolic-case1-3}) leads to
\begin{align*}
 \norm{\sigma^{\alpha,\beta,\delta}-\sigma}_{L^2(\Omega)}\le C(u) \frac{[u]^{1/2}}{\kappa_1^{3/2}}\abs{\Omega}^{1/2}&\left[   \norm{\nabla \sigma^{\alpha,\beta,\delta}}_{L^2(\Omega)}\cdot\norm{\nabla(u-u^{\alpha,\beta,\delta})}_{L^2(\Omega)}+\right.\\
&\left.\;\;\norm{\sigma^{\alpha,\beta,\delta}}_{L^2(\Omega)} \cdot \norm{\Delta(u-u^{\alpha,\beta,\delta})}_{L^2(\Omega)}\right],
\end{align*}
using $\Sigma=0$ in $\partial\Omega$ and the expression of $F$.
The desired bound is derived from $q\le \gamma:=\frac{\norm{\sigma}_{L^{\infty}}}{\sigma_{\min}}$ (see Lemma 6(ii) in \cite{Richter}) and the bound on $u[\sigma,h](x)$ by \eqref{liu24-bound_u}.
The proof is complete.
\end{proof}

Combining  \eqref{Liu24-error-estimate-beta}, Lemma \ref{lem2-1} and Lemma \ref{lem2} together, we obtain the following estimate.
\begin{them}
Under the same conditions of Lemma \ref{lem2}, our reconstruction scheme constituted by \eqref{Liu24-regular-scheme-h} and \eqref{liu24-regular-system4-delta} for optimal regularizing $\alpha=O(\delta^{1/(n/2+\lambda_0)})$ and $\beta>0$
has the estimate
\begin{align}
\norm{\sigma^{\alpha,\beta,\delta}-\sigma}_{L^2(\Omega)}&\le
C(\norm{\nabla (u^{\alpha,\beta,\delta}-u)}_{L^2(\Omega)}+\norm{\Delta (u^{\alpha,\beta,\delta}-u)}_{L^2(\Omega)})
\nonumber\\&\le
\norm{\varphi^{\alpha,\beta,\delta}-\varphi}_{L^2(\Omega\times\partial\Omega)}
\nonumber\\&\le
\frac{C(\Omega,U|_{\partial \Omega},\sigma|_{\partial\Omega})}{\sqrt{\beta}}\delta^{\lambda_0/(n/2+\lambda_0)}+
  \norm{\mathcal{R}_{\beta}\mathbb{K}\varphi-\varphi},
\end{align}
where $C=C(\mathbb{A},\Omega,g,h,k_1)$.
\end{them}

\section{Discretization for linear systems}
In the previous section, we propose to recover $\sigma(x)|_{\Omega}$ via solving two linear problems \eqref{liu24-regular-system4} and \eqref{liu24-sigma-system5}. These linear systems both involve weak singular boundary and volume integrals. In this section, we will propose an efficient scheme to  handle these singular integrals numerically.
The discrete version of system \eqref{liu24-system3} will be proposed and then the dual operators in regularizing system \eqref{liu24-regular-system4} can be approximated by dual matrix of the discrete version in \eqref{liu24-system3}.
The details for the treatments of boundary integrals can be found in \cite{Kress}. For volume integrals, we apply the dual reciprocity method (DRM) and the radial integration method (RIM) which transform the volume integrals into boundary integrals, with the attractive feature that the weak singularity can be removed explicitly and the choice of internal nodes can be random, i.e., it is a domain discretization-free method.

We firstly handle $\mathbb{K}_{11}^{\Omega\to\Omega}[\tilde\mu](x)$ in \eqref{liu24-system3} via DRM, which converts the volume integral with fundamental solution as kernel function into a boundary integral by expanding the integrand in terms of some base functions $\{f_{k}:k=1,\cdots, M\}$. More precisely, $\tilde \mu$ is expanded as
\begin{equation}\label{Liu24-mu-expand}
  \tilde \mu(x)\approx \sum\limits_{k=1}^{M}\alpha_k f_k(x)
\end{equation}
where $\{f_{k}:k=1,\cdots,M\}$ are the base functions satisfying
\begin{equation}\label{Liu24-basis-function}
  f_k(x)=\Delta\hat f_k(x)
\end{equation}
for some known $\hat f_k(x)$ corresponding to internal nodes $x_k\in\Omega_{\epsilon}$, and $\alpha_k$ are the expansion coefficients to be determined. Substituting \eqref{Liu24-mu-expand} and \eqref{Liu24-basis-function} into $\mathbb{K}_{11}^{\Omega\to\Omega}[\tilde\mu](x)$ in \eqref{liu24-system3} yields
\begin{align}\label{Liu24-DRM-K11}
\mathbb{K}_{11}^{\Omega \to \Omega}[\tilde\mu](x)
=\int_\Omega \Phi(x,y)\tilde\mu(y)dy
\approx
\sum_{k=1}^{M}\alpha_k D_k(x) ,  \quad x\in\Omega.
\end{align}
with  $D_k(x)=\int_{\Omega}\Phi(x,y)\Delta_y\hat f_k(y)dy $. By integrating by parts, we have
\begin{align}\label{Liu24-DRM}
D_k(x)&=  \int_{\partial \Omega} [\Phi(x,y)\partial_{\nu(y)} \hat f_k(y)
- \hat f_k(y)\partial_{\nu(y)}\Phi(x,y)] ds(y)
+\int_\Omega\hat f_k(y)\Delta_y\Phi(x,y)dy\nonumber\\
&=\int_{\partial \Omega}[\Phi(x,y)\partial_{\nu(y)} \hat f_k(y)
- \hat f_k(y)\partial_{\nu(y)}\Phi(x,y)] ds(y)
-\hat f_k(x), \quad x\in\Omega.
\end{align}
It can be seen that $D_k(x)$ only involves surface integral and the weak singularity in $\mathbb{K}_{11}^{\Omega \to \Omega}[\tilde\mu](x)$ has been integrated explicitly.

However, for volume integral $\mathbb{K}_{21}^{\Omega\to\partial\Omega}[\tilde\mu](x)$ in \eqref{liu24-system3}, since its kernel function is $\partial_{\nu(x)}\Phi(x,y)$, when we apply the same arguments to transform it into surface integrals, the hyper-singular kernel  $\partial_{\nu(y)}\partial_{\nu(x)}\Phi(x,y)$ will appear. In order to overcome this problem, we apply the radial integration method (RIM) to  calculate $\mathbb{K}_{21}^{\Omega\to\partial\Omega}[\tilde\mu](x)$ efficiently.

The essence of RIM to deal with volume integral $\int_{\Omega}k(x,y)dy$ with continuous or weakly singular integrand $k(x,y)$  is to represent
$k(x,y)=\text{div}_y\vec{v}(x,y)$
for some continuously differentiable vector field $\vec{v}(x,y)$.
Then the volume integral  can be transformed into a surface integral $\int_{\partial \Omega}\vec{v}(x,y)\cdot \nu(y)ds(y)$.  For convex domain $\Omega$, it has been shown that (Theorem 3.2 and Theorem 4.1, \cite{RIM_Nintcheu}), $\vec{v}(x,y)$ related to $k(x,y)$ can be constructed as
\begin{equation}\label{Liu24-RIM-v}
  \vec{v}(x,y)=\frac{y-x}{\abs{y-x}^n}\int_{0}^{\abs{y-x}}r^{n-1}k(x,x+r\frac{y-x}{\abs{y-x}})dr,\quad y \in \overline{\Omega} \subset \mathbb{R}^n
\end{equation}
with fixed source point $x\in \mathbb{R}^{n}$. Then we have from divergence theorem that
\begin{equation}\label{Liu24-RIM}
  \int_{\Omega}k(x,y)dy=\int_{\partial \Omega}\frac{(z-x)\cdot \nu(z)}{\abs{z-x}^n}\int_{0}^{\abs{z-x}}r^{n-1}k(x,x+r\frac{z-x}{\abs{z-x}})dr ds(z),\quad x \in \mathbb{R}^n.
\end{equation}
That is, RIM transforms the volume integral into $(n-1)$-dimensional surface integral and one-dimensional radial integral from source point $x$ to boundary point $z$. Usually, the radial integral in \eqref{Liu24-RIM} cannot be integrated explicitly especially for $k(x,y)$ in general form. The Gaussian quadrature may be effectively used to numerical evaluate the radial integral in \eqref{Liu24-RIM}. Then, under the variable transform
$r=\frac{\abs{z-x}}{2}(1+t)$ for $t\in [-1,1]$,
\eqref{Liu24-RIM} can be rewritten as
\begin{equation}\label{Liu24-RIM1}
  \int_{\Omega}k(x,y)dy=\int_{\partial \Omega}\frac{(z-x)\cdot \nu(z)}{2^n}\int_{-1}^{1}(1+t)^{n-1}k(x,x+\frac{1+t}{2}(z-x))dt\;ds(z), \; x \in \mathbb{R}^n.
\end{equation}

Using \eqref{Liu24-RIM1} for
$$k(x,y):=2\partial_{\nu(x)}\Phi(x,y)\tilde \mu(y)=\frac{2(y-x)\cdot\nu(x)\tilde\mu(y)}{2^{n-1}\pi\abs{y-x}^n},$$
we then have
\begin{align}
    &\mathbb{K}_{21}^{\Omega \to \partial\Omega}[\tilde\mu](x)\nonumber\\
    =&2\int_{\partial \Omega}\frac{(z-x)\cdot \nu(z)}{2^{n}}\int_{-1}^{1}(1+t)^{n-1}\left(\frac{\frac{1+t}{2}(z-x)\cdot\nu(x)}{2^{n-1}\pi\abs{\frac{1+t}{2}(z-x)}^n}\right)\tilde\mu(x+\frac{1+t}{2}(z-x))dt ds(z)\nonumber\\
    =&\int_{\partial \Omega}\frac{(z-x)\cdot \nu(z)\times(z-x)\cdot \nu(x)}{2^{n-1}\pi\abs{z-x}^{n}}\int_{-1}^{1}\tilde\mu(x+\frac{1+t}{2}(z-x))dt ds(z), \qquad x\in\partial\Omega. \label{Liu24-K211}
\end{align}
Define $RK(x,z):=\frac{(z-x)\cdot \nu(z)\times(z-x)\cdot \nu(x)}{2^{n-1}\pi\abs{z-x}^{n}}$ and apply the Gaussian quadrature to radial integral in \eqref{Liu24-K211} with respect to $t$, we have
\begin{equation}\label{Liu24-K211G}
  \mathbb{K}_{21}^{\Omega \to \partial\Omega}[\tilde\mu](x)\approx
  \int_{\partial \Omega}RK(x,z)\sum\limits_{l=1}^{N_G}w_l\tilde\mu(x+\frac{1+t_l}{2}(z-x)) ds(z), \quad x\in\partial\Omega,
\end{equation}
where $w_l$ and $t_l$ for $l=1,\cdots,N_G$ are $N_G$-order Gauss quadrature weights and points, respectively.

For $\partial \Omega\in C^2$, by Lemma 6.16 in \cite{Kress}, there exists a constant $L>0$ such that $\abs{(z-x)\cdot \nu(x)}\le L \abs{z-x}^2$
for all $x,z\in\partial \Omega$. So it follows that
\begin{align*}
    \abs{RK(x,z)}
    \le\frac{\abs{(z-x)\cdot \nu(z)} \times \abs{(z-x)\cdot \nu(x)}}{2^{n-1}\pi\abs{z-x}^{n}}
    \le \frac{L^2}{2^{n-1}\pi}\abs{z-x}^{4-n},\quad x,z \in\partial \Omega.
\end{align*}
That is, $RK(x,z)$ is essentially a continuous function  with the expression
\begin{align}\label{Liu24-RK211}
    RK(x,z)=
    \begin{cases}
\frac{(z-x)\cdot \nu(z)\times(z-x)\cdot \nu(x)}{2^{n-1}\pi\abs{z-x}^{n}}, &\quad x\neq z,\\
0,&\quad x=z.
\end{cases}
\end{align}
In this way, the volume integral $\mathbb{K}_{21,1}^{\Omega \to \partial \Omega}[\tilde\mu](x)$ is converted into surface integral, and the weak singularity have been removed completely by RIM. Expanding $\tilde \mu$ via \eqref{Liu24-mu-expand}, we finally have
\begin{equation}\label{Liu24-K211G-Mu}
  \mathbb{K}_{21}^{\Omega \to \partial\Omega}[\tilde\mu](x)\approx
  \sum\limits_{k=1}^{M}\alpha_k\int_{\partial \Omega}RK(x,z)\sum\limits_{l=1}^{N_G}w_l f_k(x+\frac{1+t_l}{2}(z-x)) ds(z), \quad x\in\partial\Omega.
\end{equation}
Since all the domain integrals in \eqref{liu24-system3} have been transformed into surface integrals via DRM and RIM, the numerical recovery of $\sigma(x)$ can be generated efficiently, noticing that the linear system for the density functions involves only the values on the boundary grids. Such a scheme will decrease the computational cost, especially in 3-dimensional cases.

Here a discrete scheme in the case $\Omega\subset\mathbb{R}^2$ is given only, where the Nystr$\ddot{\text{o}}$m scheme is applied to construct the quadrature rules for boundary integrals. As for $\Omega\subset \mathbb{R}^3$,  the scheme stated in \cite{Colton1} (Chapter 3.7) builds an efficient scheme for computing the integrals with smooth integrals and can be used  to handle the singularities.

Assume that the boundary curve $\partial \Omega$ has a $2\pi$-periodic parametric representation
$$\partial\Omega:=\{x(t)=(x_{1}(t),x_{2}(t)),\; t\in [0,2\pi]\}.$$

Introduce the quadrature points $t_j=\pi j/\tilde n,j=0,1,\cdots,2\tilde n-1$ for boundary $\partial\Omega$, we firstly handle surface integral
$\mathbb{K}_{22}^{\partial\Omega\to\partial\Omega}[\tilde\psi](x)$ with weak-singular kernel in linear system \eqref{liu24-system3}, which
has the parametric representation
$$
\mathbb{K}_{22}^{\partial\Omega\to\partial\Omega}[\tilde\psi](x(t))=\frac{1}{2\pi}
\int_{0}^{2\pi}2\partial_{\nu(x)}\Phi(x(t),x(\tau))\tilde \psi( x(\tau))|x'(\tau)|d\tau, \qquad t\in[0,2\pi].
$$
Defining a smooth kernel function $\mathbf{K}_{22}(t,\tau)$ as
\begin{eqnarray}\label{liu24-K22,1}
\mathbf{K}_{22}(t,\tau)=\frac{-1}{2\tilde n}\partial_{\nu(x(t))}\left(\ln\abs{x(t)-x(\tau)} \right)|x'(\tau)|=
\begin{cases}
\frac{1}{2\tilde n}\frac{x''(t)\cdot\nu(x(t))}{2|x'(t))|}, &t=\tau,\\
\frac{(x(\tau)-x(t))\cdot \nu(x(t))}{2\tilde n |x(t)-x(\tau)|^2}|x'(\tau)|,&\hbox{elsewhere}.
\end{cases}
\end{eqnarray}
Then it follows in terms of Nystr$\ddot{\text{o}}$m quadrature rule that
\begin{equation}\label{Liu24-discrete-K22,1}
  \mathbb{K}_{22}^{\partial\Omega\to\partial\Omega}[\tilde\psi](x(t))=2\int_{\partial\Omega}
\partial_{\nu(x(t))}\Phi(x(t),y)\tilde\psi(y)dy\approx2
\sum\limits_{j=1}^{2\tilde n}\mathbf{K}_{22}(t,t_j)\tilde\psi(x(t_j)).
\end{equation}

Now we consider the discrete version of domain integral.
Since $D_k(x)$ given by \eqref{Liu24-DRM} has smooth kernel function, the discrete version of $\mathbb{K}_{11}^{\Omega\to\Omega}[\tilde\mu](x)$ for $x\in \Omega$ in terms of Nystr$\ddot{\text{o}}$m is
\begin{align}\label{liu24-k11-discrete}
\mathbb{K}_{11}^{\Omega\to\Omega}[\tilde\mu](x)
&\approx \sum_{k=1}^{M}\alpha_k \left [ \sum_{j=1}^{2\tilde n} \left(  \mathbf{G}_{j}(x)\hat q_{jk}-\mathbf{H}_{j}(x)\hat f_{jk}\right)-\hat f_{k}(x)\right ], \quad x\in \Omega,
\end{align}
where $\hat f_{jk}=\hat f_k(y_j)$,
 $\hat q_{j k}= \left.\partial_{\nu(y)} \hat f_k(y)\right | _{y=y_j} $ with $y_j=x(t_j)$ and
\begin{align*}
     \mathbf{G}_{j}(x)=\frac{1}{2\tilde n}\ln\frac{1}{|x-y_j|}|x'(t_j)|, \qquad
     \mathbf{H}_{j}(x)=\frac{1}{2\tilde n}\frac{(x-y_j)\cdot \nu(y_j)}{|x-y_j|^2}|x'(t_j)|.
\end{align*}
For $\mathbb{K}_{21}^{\Omega \to \partial\Omega}[\tilde\mu](x)$ in \eqref{Liu24-K211G-Mu}, since it only involves smooth kernel,  it can be approximated by
\begin{equation}\label{Liu24-discrete-K21,1}
  \mathbb{K}_{21}^{\Omega \to \partial\Omega}[\tilde\mu](x(t))\approx \sum\limits_{k=1}^{M}\alpha_k \sum\limits_{j=1}^{2\tilde n}\mathbf{RK}^k_j(t)
\end{equation}
in parametric representation with $$\mathbf{RK}^k_j(t):=\frac{1}{2\tilde n}RK(x(t),y_j)\sum\limits_{l=1}^{N_G}w_l f_k(x(t)+\frac{1+t_l}{2}(y_j-x(t)))|x'(t_j)|.$$

Finally, we yield the following approximate version of \eqref{liu24-system3} from \eqref{Liu24-discrete-K22,1}-\eqref{Liu24-discrete-K21,1}:
\begin{align}\label{Liu-24-discrete}
    \begin{cases}
    \sum\limits_{k=1}^{M}\alpha_k \mathbf{DK}(x;k)+\sum\limits_{j=1}^{2\tilde n} \mathbf{G}_j(x)\tilde \psi_j=U(x),\\
    \sum\limits_{k=1}^{M}\alpha_k \sum\limits_{j=1}^{2\tilde n}\mathbf{RK}^k_j(t)+\tilde \psi (x(t))+2\sum\limits_{j=1}^{2\tilde n}\mathbf{K}_{22}(t;t_j)\tilde \psi_j =2\tilde g(x(t))-2\tilde h(x(t))U(x(t))
    \end{cases}
\end{align}
for the unknowns $\alpha_k (k=1,\cdots,M)$ and $\tilde \psi_j:=\tilde \psi(x(t_j)) (j=1,\cdots,2\tilde n)$, where
$$\mathbf{DK}(x;k)=\sum\limits_{j=1}^{2\tilde n} \left(  \mathbf{G}_{j}(x)\hat q_{jk}- \mathbf{H}_{j}(x)\hat f_{jk}\right)-\hat f_{k}(x)$$

By specifying $(x,t)$ in \eqref{Liu-24-discrete} at the collocation points $(x
_k,t_j)\in \Omega_{\epsilon}\times [0,2\pi]$ with $k=1,\cdots,M$ and $j=1,\cdots,2\tilde n$, we can finally solve $\alpha_k$ for $k=1,\cdots,M$ and $\tilde \psi_j$ for $j=1,\cdots,2\tilde n$. Then the density function $\tilde \mu(x)$ can be approximated by \eqref{Liu24-mu-expand}.

The special basis function $f_k(x)$ is needed to expand density function $\tilde \mu(x)$. Since $\tilde f_k(x)$ is required to meet $\Delta\tilde f_k(x)=f_k(x)$, it is convenient to choose $f_k(x)$ in some special form. For specified internal grid $x_k\in\Omega_{\epsilon}\subset \mathbb{R}^n$, as recommended in \cite{DRM_Brunton}, a typical form is
\begin{equation}\label{liu11-01-02}
f_k(x)=1+r_k(x):=1+|x-x_k|, \quad k=1,\cdots,M,
\end{equation}
where $r_k(x)$ is the distance between the field point $x$ and the internal grid $x_k$.

For radial basis function $f_k(x)$ in this form,
the function $\hat f_k(x)$ to  \eqref{Liu24-basis-function} can be chosen as
\begin{eqnarray}
\hat f_k(x)=
\begin{cases}
\frac{r_k^2(x)}{4}+\frac{r_k^3(x)}{9}, &n=2,\\
\frac{r_k^2(x)}{6}+\frac{r_k^3(x)}{12}, &n=3.
\end{cases}
\end{eqnarray}

\section{Numerical experiments }

We test the performances of our scheme for recovering  $(h,\sigma)$ in \eqref{liu11-01} by two examples in $\Omega \subset \mathbb{R}^2, \mathbb{R}^3$, showing the validity of our proposed schemes.

Assume that $\sigma|_{\partial\Omega}$ is known.  The noisy measurement data  are simulated in the way
$$
U^{\delta}(x)=u(x)+\delta\times \text{rand}(x), \quad x\in\Omega_{\epsilon},
$$
where $\delta>0$ is noise level and $\text{rand}(x)\in (-1,1)$ is random number obeying uniform distribution.

{\bf Example 1. A 2-dimensional model.} Consider the case that $\Omega$ is a heart-shaped domain, with boundary curve in polar coordinates represented as
\begin{equation*}
\partial\Omega:\equiv\{x(t)=(0.5\cos t,0.5\sin t-0.4\sin^2 t+0.4),\; t\in [0,2\pi]\}+(1,0.8).
\end{equation*}
We take the conductivity in $\overline{\Omega}$ as
\begin{equation}\label{Liu24-sigma2D-expression}
   \sigma(x,y)=(1+x+y)^2,
\end{equation}
as shown in Figure \ref{Liu-24-exact-sigma2D}. The exact solution to forward problem \eqref{liu11-01} is taken as
\begin{equation}\label{Liu24-u-expression}
  u(x,y)=-\frac{x^2-y^2}{1+x+y}-2.
\end{equation}
Meanwhile the boundary impedance coefficient
\begin{equation}\label{Liu24-h2D-expression}
 h(x,y)=2x+y^2+2,\quad (x,y)\in\partial\Omega.
\end{equation}

\begin{figure}[htbp]
\centering
\begin{minipage}[t]{0.45\linewidth}
\centering
\includegraphics[scale=0.35]{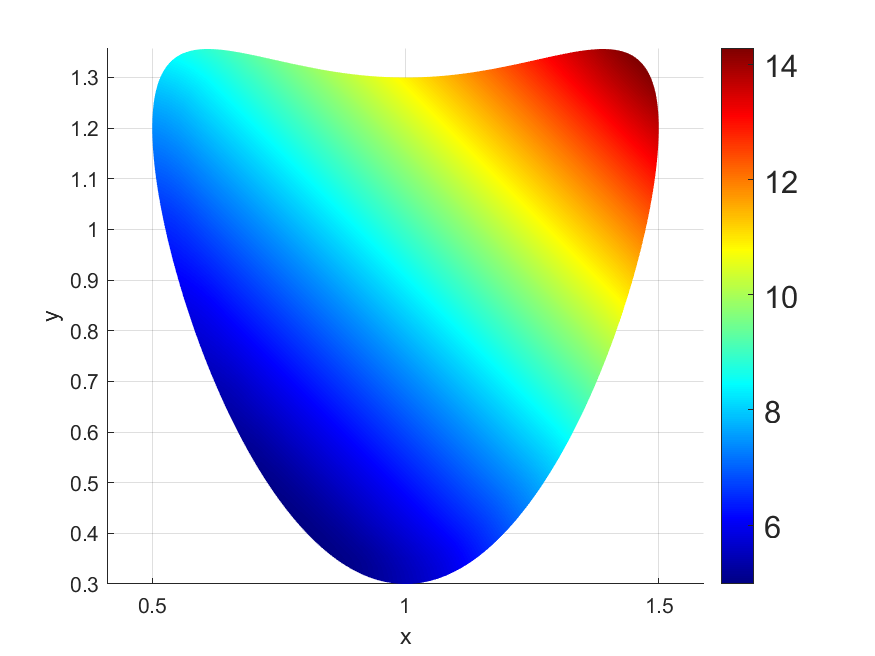}
\caption{The exact conductivity $\sigma(x,y)$.}
\label{Liu-24-exact-sigma2D}
\end{minipage}
\begin{minipage}[t]{0.45\linewidth}
\centering
\includegraphics[scale=0.35]{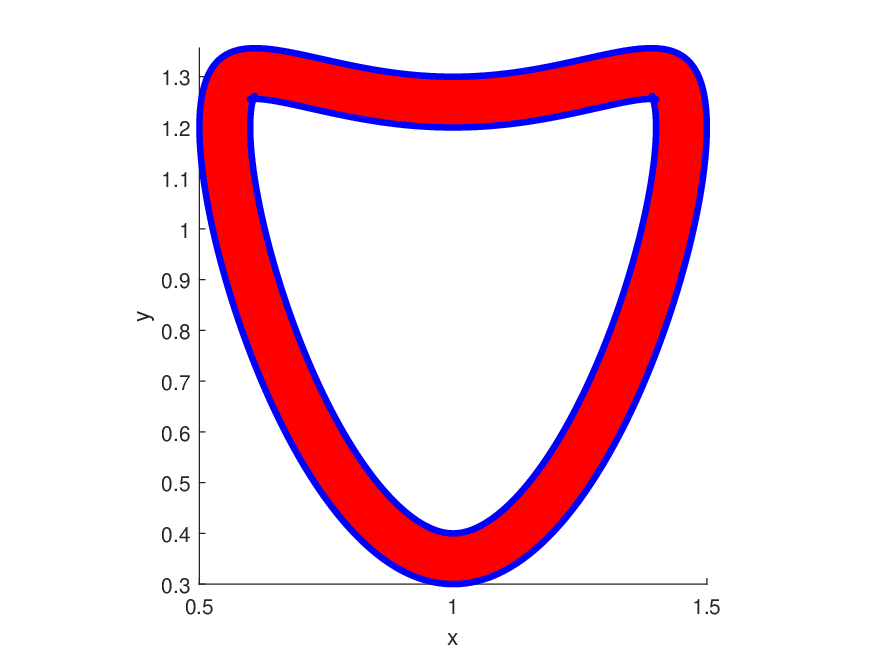}\\
\caption{The measure domain $\Omega_{1/10}.$}
\label{Liu-24-measure-domain2D}
\end{minipage}
\end{figure}

\begin{figure}[htbp]
  \centering
  \includegraphics[scale=0.6]{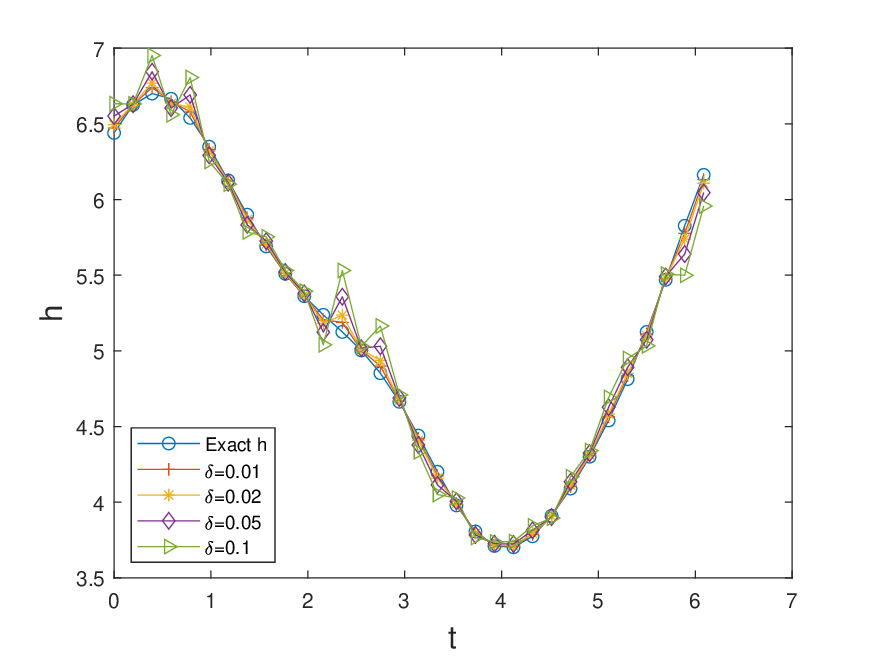}\\
  \caption{Exact $h$ and its reconstructions for different noise levels}
  \label{Liu-24-Recon-h2D}
\end{figure}

It can be verified that such a configuration for \eqref{liu11-01} meets $g(x,y)<0$ and $|\nabla u[\sigma,h,g]|_{\Omega}>0$.

We specify the internal measurable domain as $\Omega_{1/10}\equiv \{x\in\Omega,\text{dist}(x,\partial \Omega)\leq 1/10\}$ for this model, as shown in red colour in Figure \ref{Liu-24-measure-domain2D}.


We firstly recover  $h(x,y)$ by our proposed scheme. To calculate $h$ numerically, $\partial\Omega$ is divided into $512$-subintervals by discrete points $x(t_i)$ with $t_i:=i\frac{2\pi}{N_e}, i=0,\cdots,N_e-1$ for $N_e=512$. Since $u(x,y)$ is smooth enough, we take $\lambda_0=\frac{1}{2}$ in our choice strategy $\alpha(\delta)=O(\delta^{1/(1+\lambda_0)})$. Numerically we fix $\alpha(\delta)=0.6\times\delta^{1/1.5}$. In Figure \ref{Liu-24-Recon-h2D}, we show the reconstructions of $h$ given in \eqref{Liu24-h2D-expression} for different noise levels $\delta=0.01,0.02,0.05,0.1$, respectively at
$32$-boundary nodes. It can be seen that the recovered results are satisfactory. It is also noticed that the reconstructions are contaminated quickly as the noise level $\delta$ increases, especially for $\delta=0.1$.

Next we recover $\sigma$ given in \eqref{Liu24-sigma2D-expression} using noisy inversion input data $U^{\delta}(x,y)|_{\Omega_{1/10}}$ and the recovered $h^{\delta}(x,y)$ with $\delta=0.01,0.02,0.05,0.1$. We take $256$ internal nodes in measurement domain $\Omega_{1/10}$, where  $U^{\delta}(x,y)$ are simulated to recover $\sigma(x,y)$. The regularizing parameter $\beta>0$ is chosen in the strategy
$$\beta(\delta)\to 0,\quad \frac{\delta^{\lambda_0/(1+\lambda_0)}}{\sqrt{\beta}}\to 0\quad \hbox{as} \quad \delta\to 0.$$
Here we take $\beta(\delta)=0.03\times \delta^{\lambda_0/(1+\lambda_0)}$ for $\lambda_0=\frac{1}{2}$, together with the recovered Robin coefficient $h(x,y)$. It should be noticed that the number of boundary nodes (512) is much larger than that of internal nodes (256). The reason for such a setting is that all the volume integrals in \eqref{liu24-system3} have been transformed into boundary integrals via DRM and RIM schemes, which require sufficient fine boundary grids  to ensure the numerical accuracy of boundary integrals.

\begin{figure}[htbp]
\centering
\subfigure{
\begin{minipage}[t]{0.23\linewidth}
\centering
\includegraphics[scale=0.25]{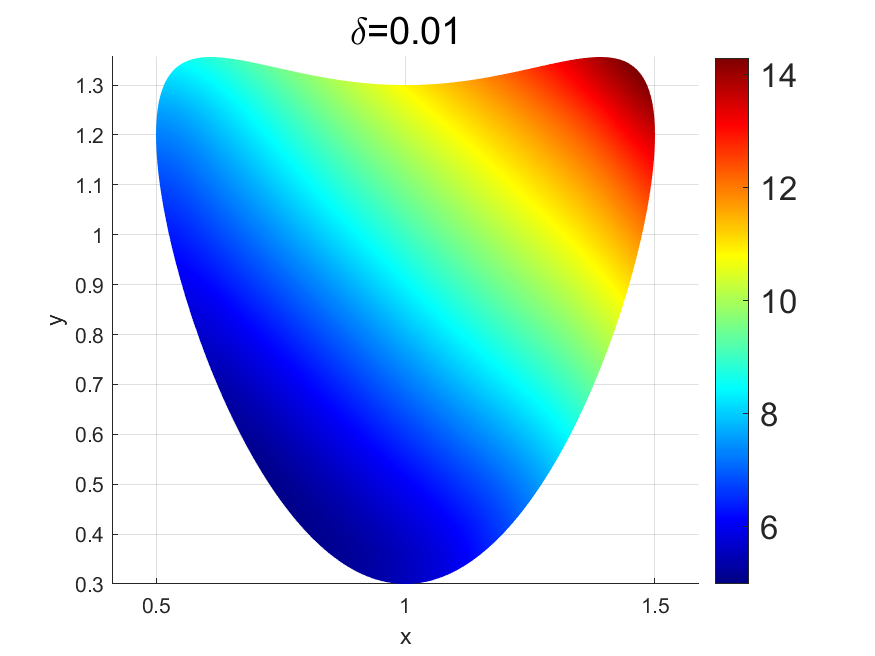}
\end{minipage}
}
\subfigure{
\begin{minipage}[t]{0.24\linewidth}
\centering
\includegraphics[scale=0.25]{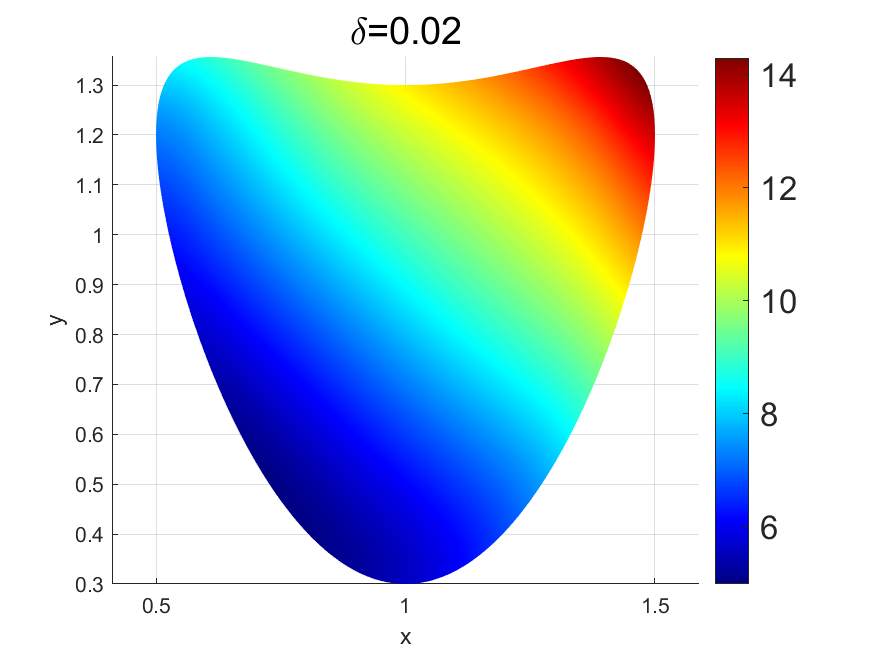}
\end{minipage}%
}%
\subfigure{
\begin{minipage}[t]{0.24\linewidth}
\centering
\includegraphics[scale=0.25]{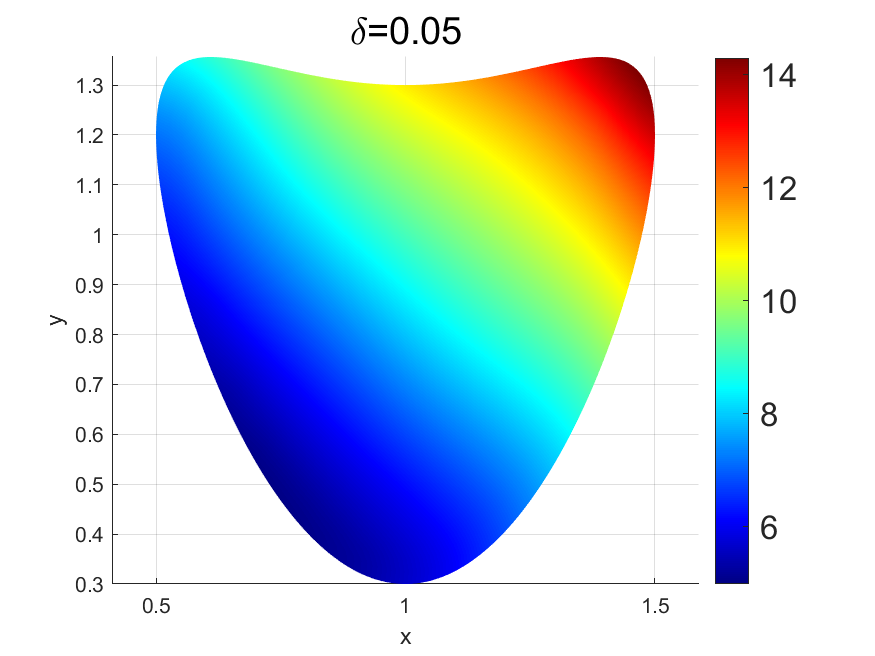}
\end{minipage}%
}%
\subfigure{
\begin{minipage}[t]{0.24\linewidth}
\centering
\includegraphics[scale=0.25]{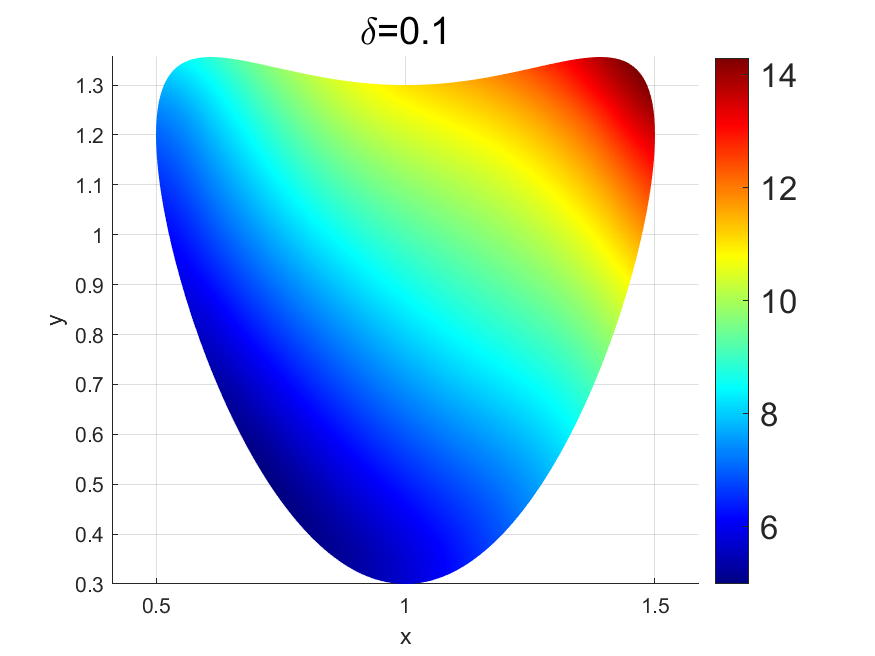}
\end{minipage}%
}%
\\
\subfigure{
\begin{minipage}[t]{0.245\linewidth}
\centering
\includegraphics[scale=0.25]{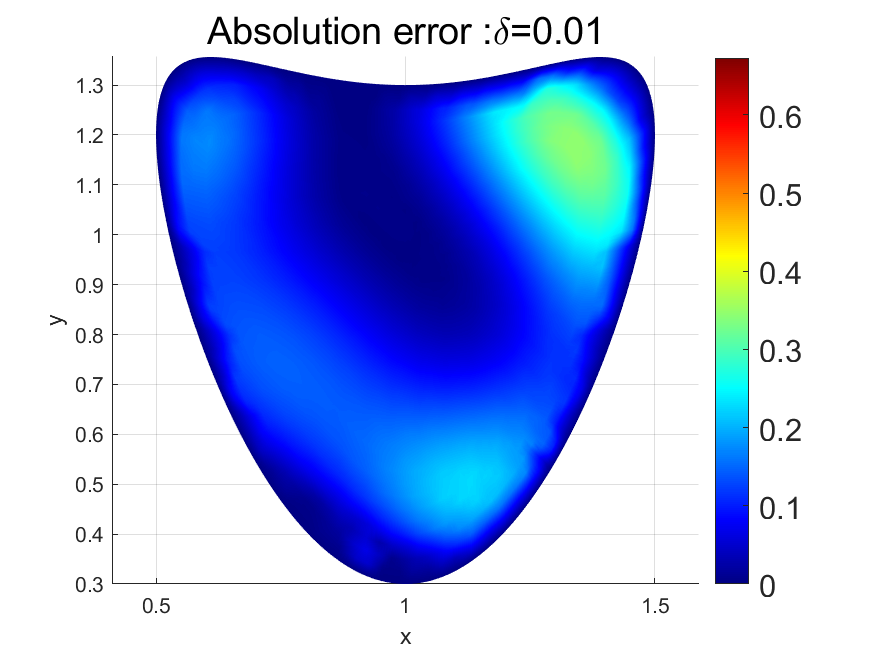}
\end{minipage}%
}%
\subfigure{
\begin{minipage}[t]{0.24\linewidth}
\centering
\includegraphics[scale=0.25]{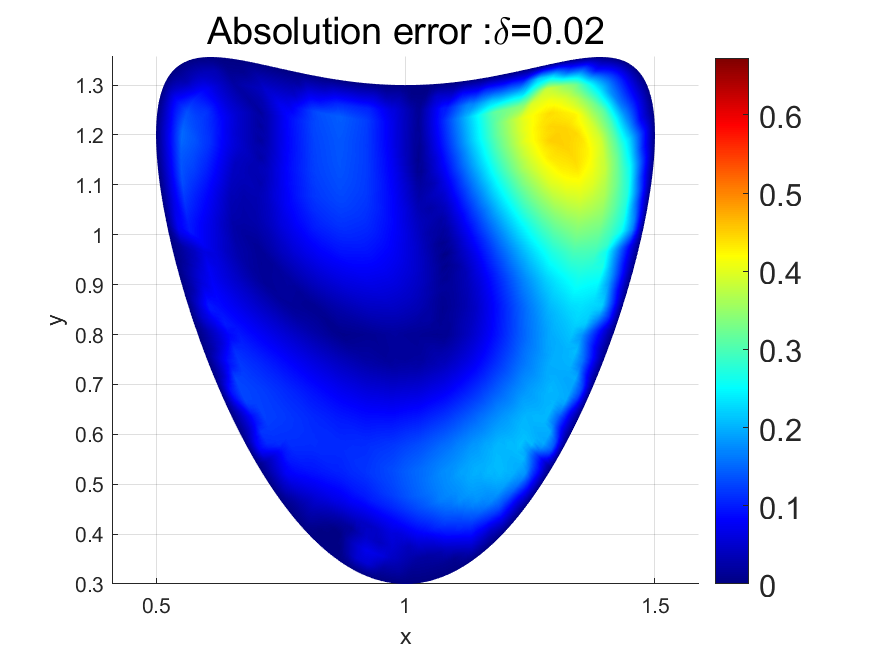}
\end{minipage}%
}%
\subfigure{
\begin{minipage}[t]{0.24\linewidth}
\centering
\includegraphics[scale=0.25]{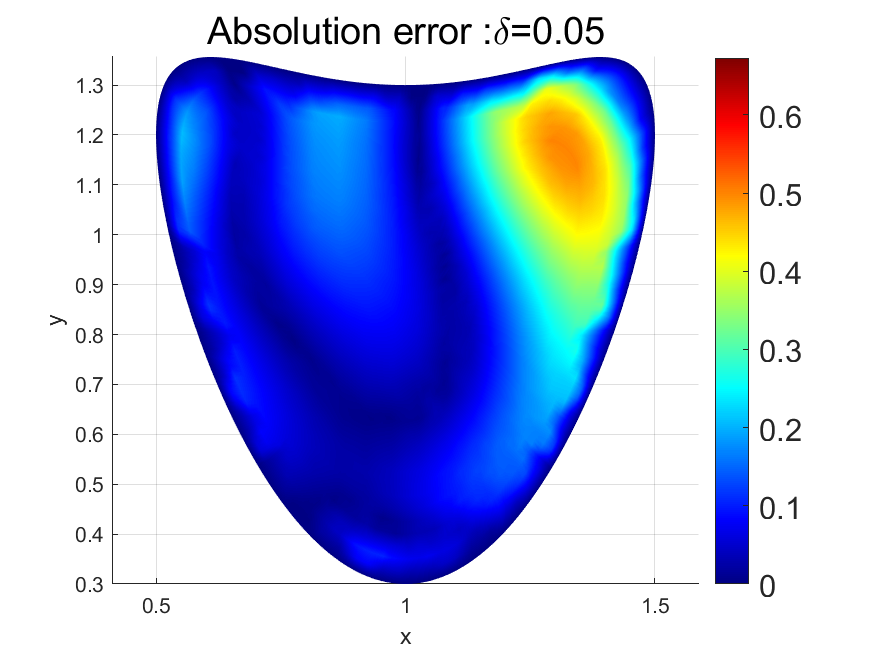}
\end{minipage}%
}%
\subfigure{
\begin{minipage}[t]{0.24\linewidth}
\centering
\includegraphics[scale=0.25]{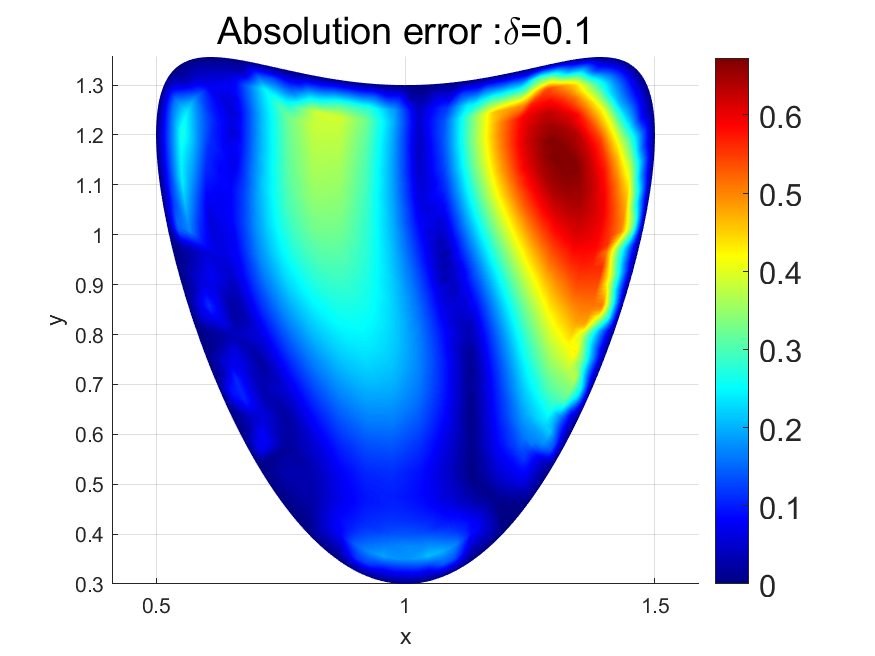}
\end{minipage}%
}%
\centering
\caption{The reconstructions of $\sigma
$ together with error distributions for different noise levels.}
\label{Liu-24-recover-sigma2D}
\end{figure}

We present the reconstruction performances of  $\sigma(x,y)$ in Figure \ref{Liu-24-recover-sigma2D} by our scheme for different noise levels. The top row illustrates the  reconstructions for different noise levels, while the bottom row shows the corresponding point-wise absolute error distributions. From left to right, it gives numerical results for noise levels $\delta=0.01,0.02,0.05,0.1$, respectively. Notice, the colour-bars for two rows are different, i.e., the same colour in top and bottom rows represents different values. Compared with exact $\sigma$ shown in Figure \ref{Liu-24-exact-sigma2D}, we see that the reconstructions are satisfactory and also contaminated  for large noise levels. For relative small $\delta=0.01,0.02$, the reconstructions are slightly different from the exact one, with maximum absolute errors less than $0.5$. For large noise levels $\delta=0.05,0.1$, the recovered results are deteriorated obviously, with the maximum absolute error larger than $0.5$. However, the internal structure of $\sigma(x,y)$ are still able to be identified even for the large noise level $\delta=0.1$.

To quantitatively show the approximation behavior for the reconstruction process by noise levels, we introduce the following  $L^2$-norm relative errors
\begin{align*}
\norm{\frac{h^{\delta}-h}{h}}_{L^2(\partial\Omega)}\approx \text{RE}_{\partial\Omega}(h,\delta):=\sqrt{\frac{2\pi}{N_e}\sum\limits_{i=1}^{N_e}\left|\frac{h_i^{\delta}-h_i}{h_i}\right|^2\abs{x'(t_i)}},\\
\norm{\frac{\sigma^{\delta}-\sigma}{\sigma}}_{L^2(\Omega)}\approx \text{RE}_{\Omega}(\sigma,\delta):=\sqrt{\sum\limits_{l=1}^{N_t}S_l\frac{1}{3}\sum\limits_{i=1}^{3}\left|\frac{\sigma_{l(i)}^{\delta}-\sigma_{l(i)}}{\sigma_{l(i)}}\right|^2},
\end{align*}
where $S_l$ is the area of triangle $\Delta_l$ with vertices $l(1),l(2),l(3)$ satisfying $\Omega\approx\bigcup _{l=1}^{N_t}\Delta_l$.

In Table \ref{Liu-24-table-2D}, we show the reconstruction errors, which cannot be observed from Figure \ref{Liu-24-Recon-h2D} and Figure \ref{Liu-24-recover-sigma2D}. It can be seen that both $\text{RE}_{\partial\Omega}(h,\delta)$ and $\text{RE}_{\Omega}(\sigma,\delta)$ increase as $\delta$ becomes larger. Since the error in recovering $h(x)$ will propagate to the recovering process for $\sigma(x)$, it can be noticed that the recovered errors of $\sigma(x,y)$  are all larger than those for recovering $h(x,y)$.
\begin{table}[H]
\centering
\caption{Error comparisons of reconstructions with respect to the noisy level $\delta$. }
\vspace{5pt}
\begin{tabular}{c|cccc}
\hline
Noise level $\delta$ & 0.01 & 0.02 & 0.05 & 0.1 \\
\hline
$\text{RE}_{\partial\Omega}(h,\delta)$ & 1.115167e-03 & 1.947534e-03  & 4.064032e-03 & 7.083656e-03\\

$\text{RE}_{\Omega}(\sigma,\delta)$ & 1.378185e-02 & 1.405683e-02 & 1.480780e-02 & 2.352229e-02 \\

\hline
\end{tabular}
\label{Liu-24-table-2D}
\centering
\end{table}


{\bf Example 2. A 3-dimensional model.}   Let $\Omega\subset \mathbb{R}^3$ be a pinched ball with  boundary
$$
\partial \Omega=\left \{q(\theta,\varphi): q(\theta,\varphi)=r(\theta,\varphi)x(\theta,\varphi), \theta \in [0,\pi], \varphi \in [0,2\pi]\right\},
$$
where $$r(\theta,\varphi)=0.5\sqrt{(1.5+0.5\cos2\varphi(\cos2\theta-1)},\;
x(\theta,\varphi)
=(\sin\theta\cos\varphi, \sin\theta\sin\varphi, \cos\theta)\in \mathbb{S}^2,$$ see Figure \ref{liu-24-Ex3_surface} for the geometric shape.

\begin{figure}[htbp]
\centering
\includegraphics[scale=0.43]{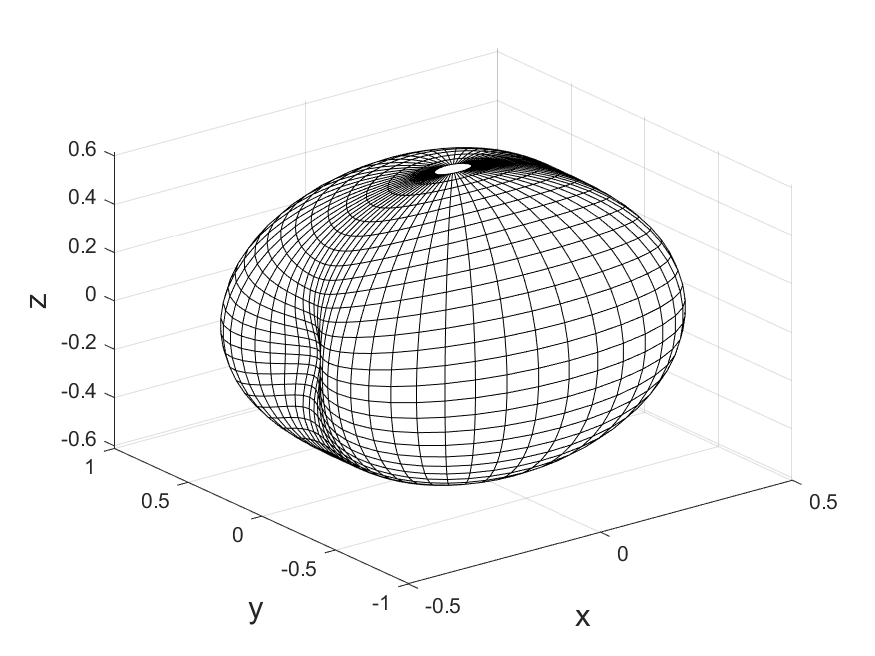}
\centering
\caption{Geometric configuration for the pinched ball $\Omega$.}
\label{liu-24-Ex3_surface}
\end{figure}

We give the conductivity distribution
\begin{align}\label{Liu-24-sigma3D-expression}
\sigma(x,y,z)=(2+0.5x+0.5y+z)^2
\end{align}
in $\Omega$ and the exact solution to forward problem
\begin{align}\label{Liu-24-u3D-expression}
u(x,y,z)=\frac{e^{x}\sin(y+\frac{\pi}{3})}{2+0.5x+0.5y+z}-5,
\end{align}
while the exact boundary impedance coefficient is
\begin{align}\label{Liu-24-h3D-expression}
h(x,y,z)=x^2+2y+z+4,
\end{align}
see Figure \ref{Liu-24-exact-h3D}. Such a configuration satisfies \eqref{liu11-01} and $g<0$.

\begin{figure}[htbp]
\centering
\begin{minipage}[t]{0.45\linewidth}
\centering
\includegraphics[scale=0.4]{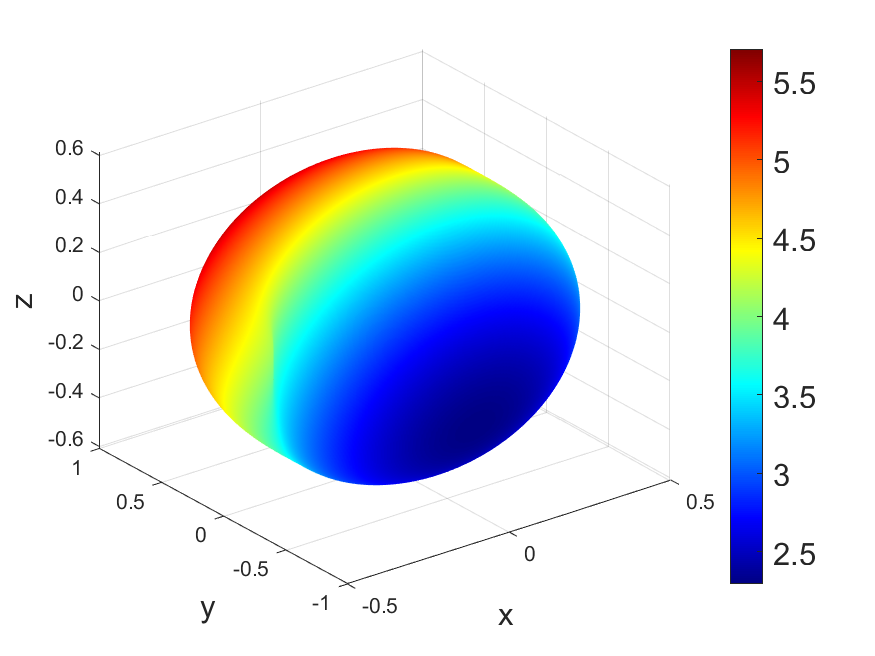}
\caption{The exact Robin coefficient $h$.}
\label{Liu-24-exact-h3D}
\end{minipage}
\begin{minipage}[t]{0.45\linewidth}
\centering
\includegraphics[scale=0.4]{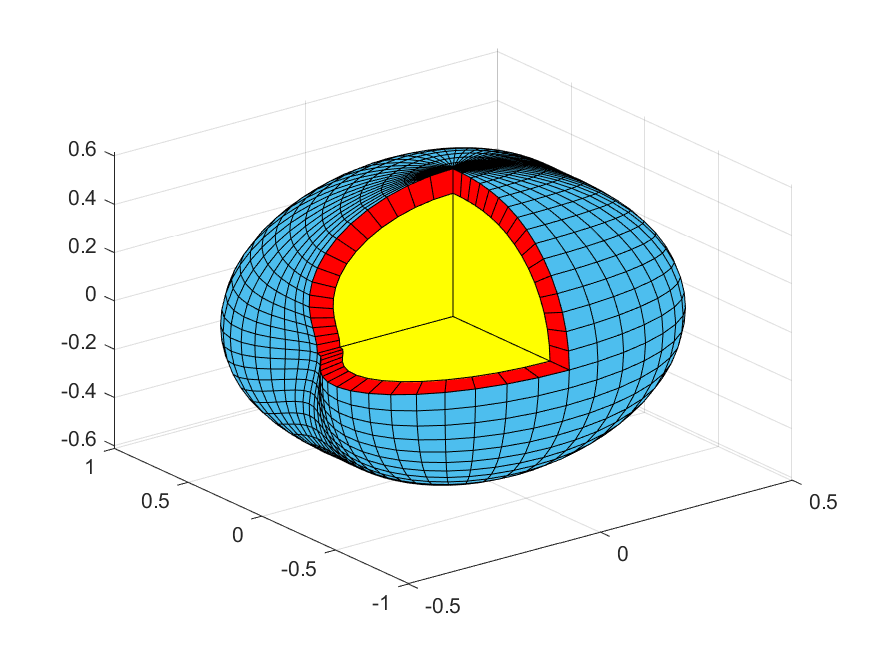}\\
\caption{The measurement domain $\Omega_{1/12}.$}
\label{Liu-24-measure-domain}
\end{minipage}
\end{figure}

\begin{figure}[htbp]
\centering
\includegraphics[scale=0.6]{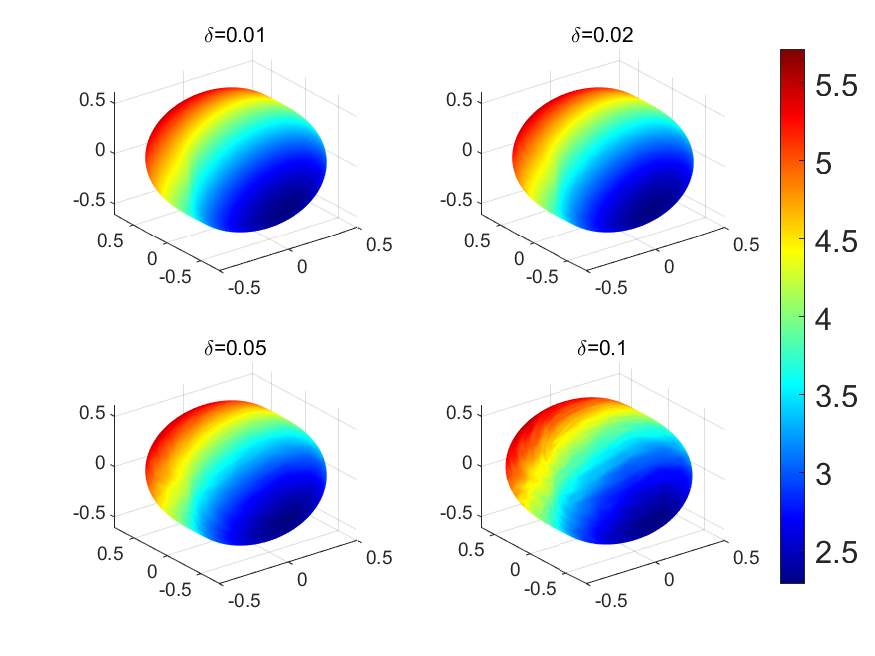}
\centering
\caption{The reconstructions of $h(x,y,z)$ with different noise levels.}
\label{Liu-24-Recon-h}
\end{figure}

We also assume that $\sigma|_{\partial\Omega}$  is known.  The measurement domain $\Omega_\epsilon$ is fixed for $\epsilon=1/12$, see the red part in Figure \ref{Liu-24-measure-domain}, while the blue and yellow parts are the surface and non-measurement domain, respectively. The noisy data are generated by
$$
U^{\delta}(x)=u(x)+\delta\times \text{rand}(x), \qquad x\in\Omega_{1/12},
$$
with $\delta>0$ and $\text{rand}(x)\in (-1,1)$  the random number obeying uniform distribution.

In our numerical scheme, the quadrature nodes on the unit sphere $\mathbb{S}^2$ are chosen in term of the polar coordinates \cite{Colton1}:
$$
x_{jk}:\equiv(\sin\theta_j\cos\varphi_k,\sin\theta_j\sin\varphi_k,\cos\theta_j),\quad j=1,\cdots,N,\quad k=0,\cdots,2N-1,
$$
where $\theta_j=\arccos t_j,\varphi_k=\pi k/N$ for $N$-order Gauss quadrature points $t_j\in [-1,1]$. Then   $2N^2$ surface nodes $q_{jk}:\equiv r(\theta_j,\varphi_k) x_{jk}\in \partial \Omega$ can be obtained to discrete $\partial\Omega$.

We firstly recover $h(x,y,z)$ based on the regularizing scheme \eqref{Liu24-regular-scheme-h} with $N=32$.  In Figure \ref{Liu-24-Recon-h}, we show the reconstructions using noisy data with  $\delta=0.01,0.02,0.05,0.1$ with $\alpha(\delta)=0.3\times\delta^{1/(1.5+\lambda_0)}$ with $\lambda_0=\frac{1}{2}$, respectively. Compared with the exact result in Figure \ref{Liu-24-exact-h3D}, it can be found clearly that the reconstructions are satisfactory and relatively stable for different noise levels. For relatively small noise levels $\delta=0.01,0.02$, the reconstructions are almost consistent with the exact result. When the noise levels are large such as $\delta=0.05,0.1$, the reconstructions are contaminated obviously, especially for $\delta=0.1$. However, the distribution of $h$ on $\partial \Omega$ can still be clearly distinguished, even for large noise level $\delta=0.1$.

Now, we continue to recover $\sigma(x,y,z)$  from  $U^{\delta}(x,y,z)$ with noisy level $\delta=0.01,0.02,0.05,0.1$ by our parameter choice strategy $\beta(\delta)=0.2\times \delta^{\lambda_0/(1.5+\lambda_0)}$ with $\lambda_0=\frac{1}{2}$. We apply the recovered $h(x,y,z)$ in the above step and $U^{\delta}(x,y,z)$ evaluated at 840 discrete points  in $\Omega_{1/12}$. Figure \ref{liu-24-exact-sigma3D} shows the exact conductivity $\sigma(x,y,z)$ given in \eqref{Liu-24-sigma3D-expression}, while  the conductivity in $\partial\Omega$ is presented in Figure \ref{liu-24-exact-sigma3D}(a) , which is assumed to be known in our reconstruction process.  The conductivity inside $\Omega$ is shown in Figure \ref{liu-24-exact-sigma3D}(b) for which we aim to recover.

\begin{figure}[htbp]
\centering
\subfigure[The conductivity in $\partial\Omega$.]
{
\begin{minipage}[t]{0.45\linewidth}
\centering
\includegraphics[scale=0.32]{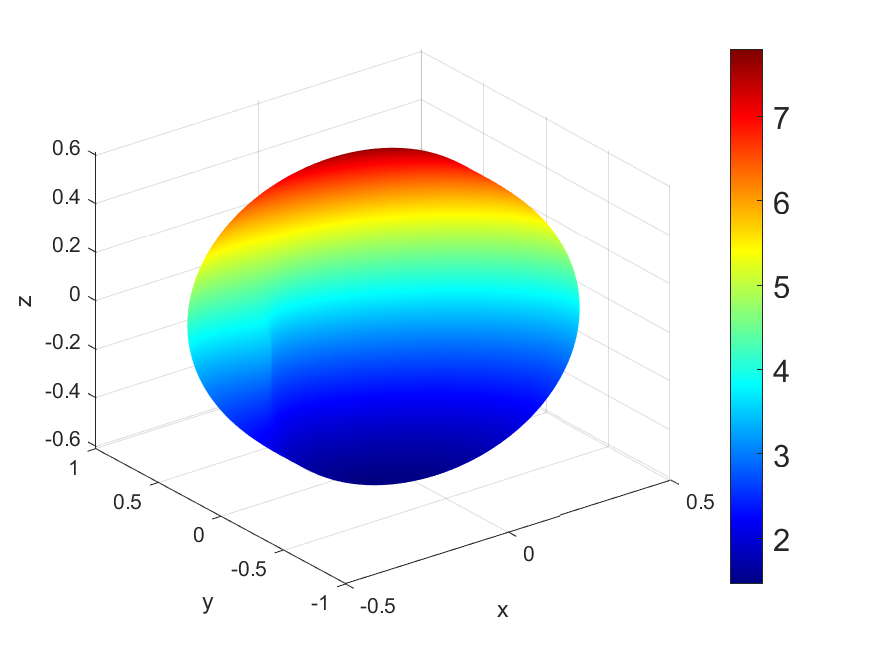}
\end{minipage}
}
\subfigure[The conductivity in $\Omega$.]
{
\begin{minipage}[t]{0.45\linewidth}
\centering
\includegraphics[scale=0.32]{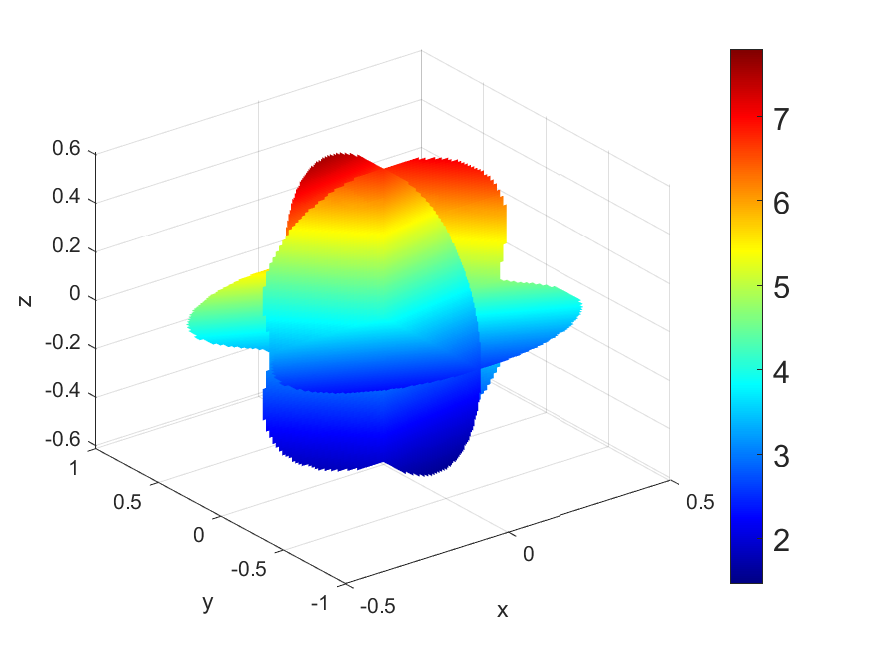}
\end{minipage}%
}%
\centering
\caption{The exact conductivity $\sigma(x,y,z)$ in $\overline\Omega$.}
\label{liu-24-exact-sigma3D}
\end{figure}

In Figure \ref{Liu-24-recover-sigma3}, we show the numerics of recovering $\sigma(x,y,z)$ in $\Omega$ for different noise levels, respectively. From first to last column, it gives the reconstructions for noise levels $\delta=0.01,0.02,0.05,0.1$, respectively.
The top row illustrates the numerical reconstructions  at three slice planes orthogonal to $x,y,z$ axis, while the bottom row shows the point-wise absolute error distributions corresponding to top row. From these figures, we can see that the reconstructions are satisfactory and also deteriorated for large noise levels. For small noise levels $\delta=0.01,0.02$, the maximum absolute error of small noise levels are both less than $0.5$ and the reconstructions are only slightly different from the exact one as shown in  Figure \ref{liu-24-exact-sigma3D}(b). However, for large noise levels $\delta=0.05,0.1$, the recovered results are contaminated obviously and the maximum absolute error of large noise levels are both greater than $0.5$. But the structure information of  $\sigma(x,y,z)$ in different parts of $\Omega$ can still be distinguished.

\begin{figure}[htbp]
\centering
\subfigure{
\begin{minipage}[t]{0.23\linewidth}
\centering
\includegraphics[scale=0.25]{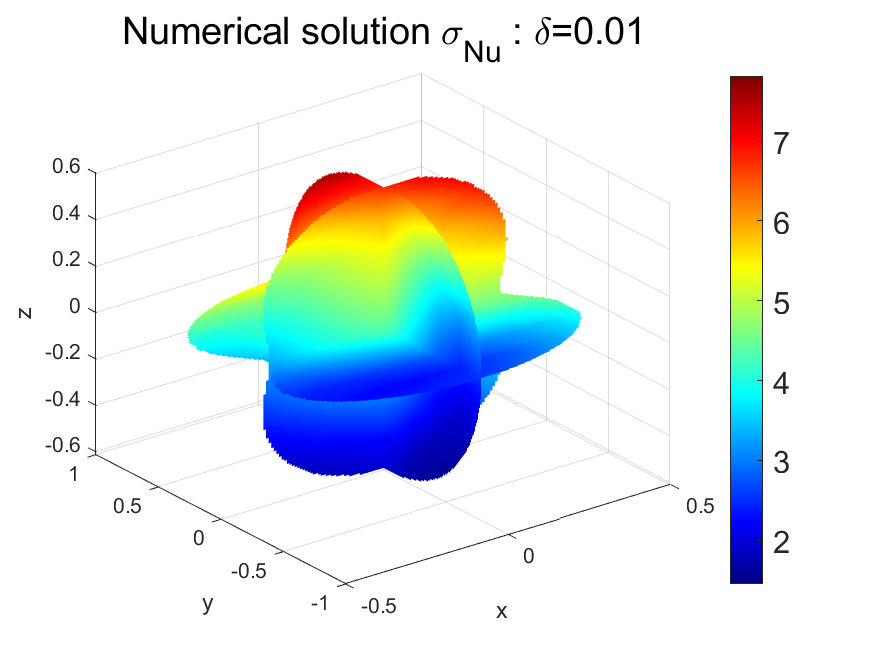}
\end{minipage}
}
\subfigure{
\begin{minipage}[t]{0.24\linewidth}
\centering
\includegraphics[scale=0.25]{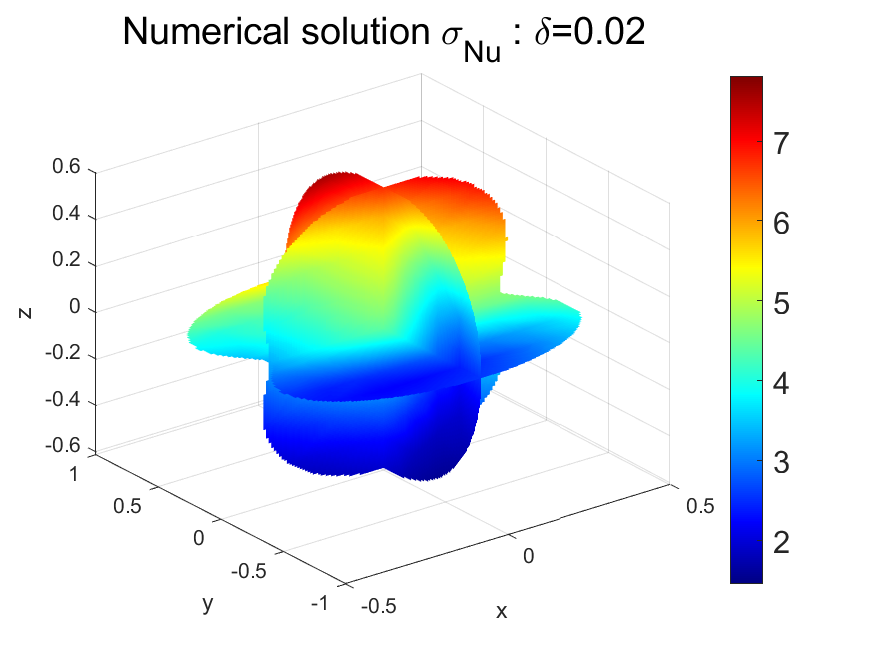}
\end{minipage}%
}%
\subfigure{
\begin{minipage}[t]{0.24\linewidth}
\centering
\includegraphics[scale=0.25]{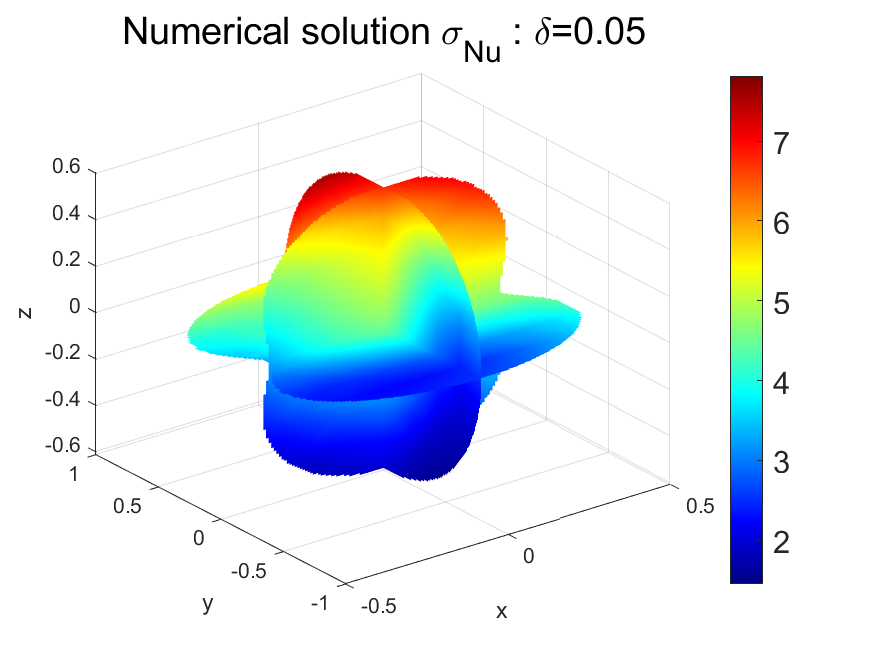}
\end{minipage}%
}%
\subfigure{
\begin{minipage}[t]{0.24\linewidth}
\centering
\includegraphics[scale=0.25]{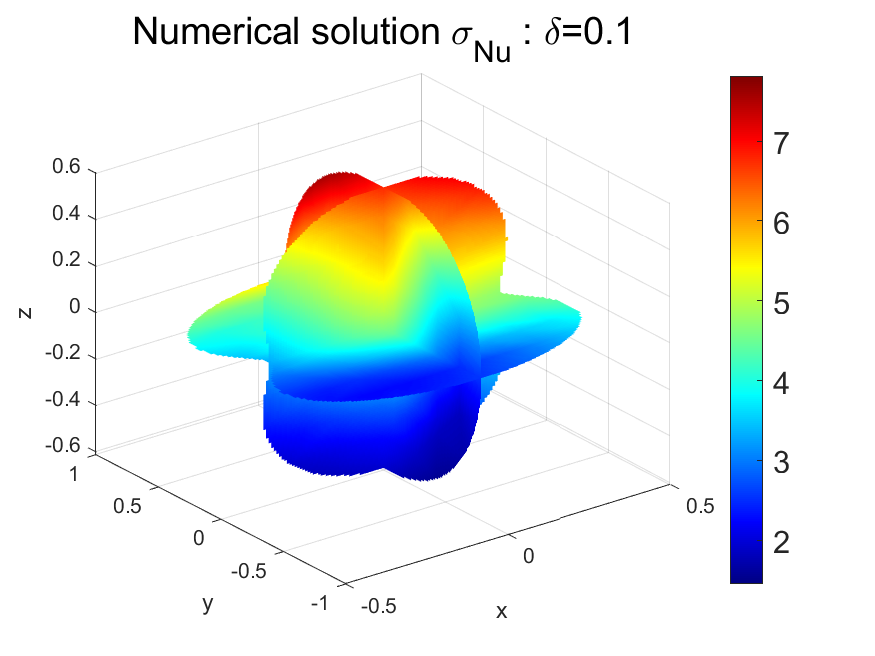}
\end{minipage}%
}%
\\
\subfigure{
\begin{minipage}[t]{0.245\linewidth}
\centering
\includegraphics[scale=0.25]{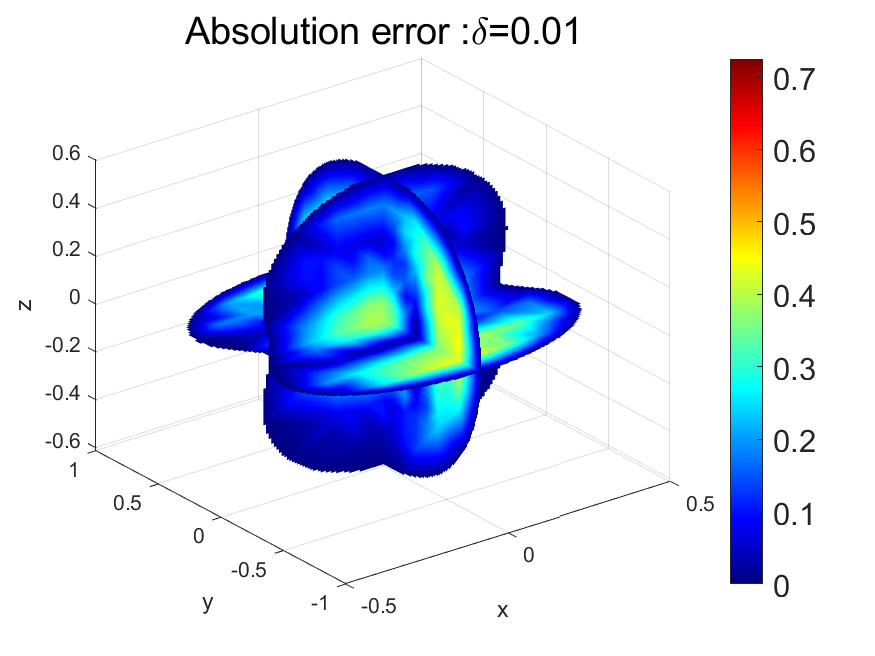}
\end{minipage}%
}%
\subfigure{
\begin{minipage}[t]{0.24\linewidth}
\centering
\includegraphics[scale=0.25]{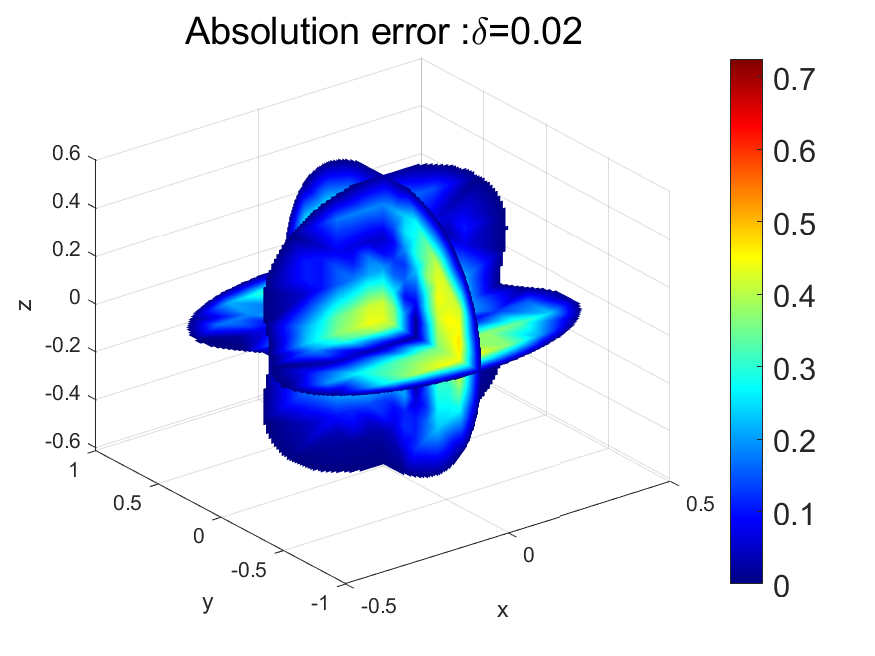}
\end{minipage}%
}%
\subfigure{
\begin{minipage}[t]{0.24\linewidth}
\centering
\includegraphics[scale=0.25]{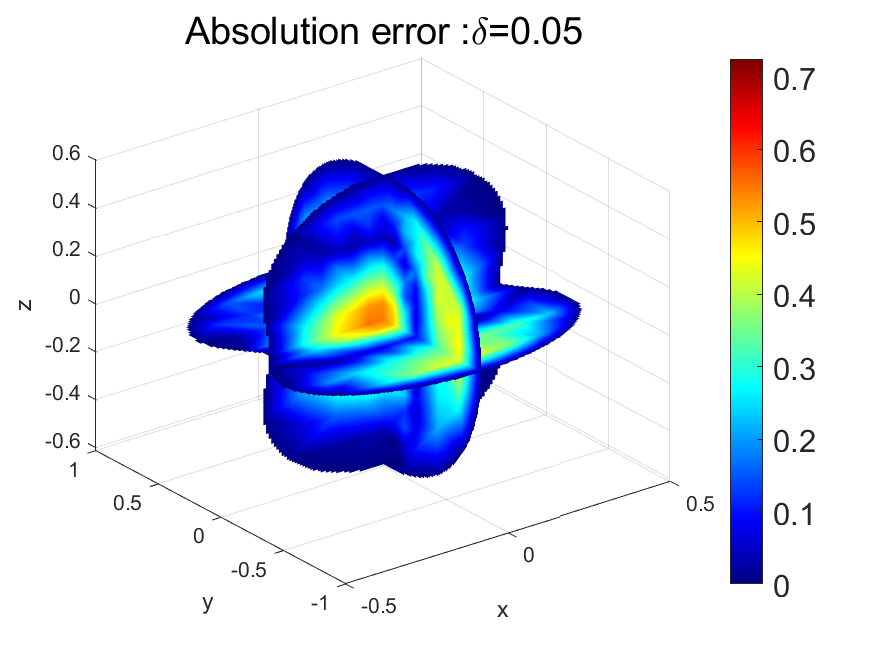}
\end{minipage}%
}%
\subfigure{
\begin{minipage}[t]{0.24\linewidth}
\centering
\includegraphics[scale=0.25]{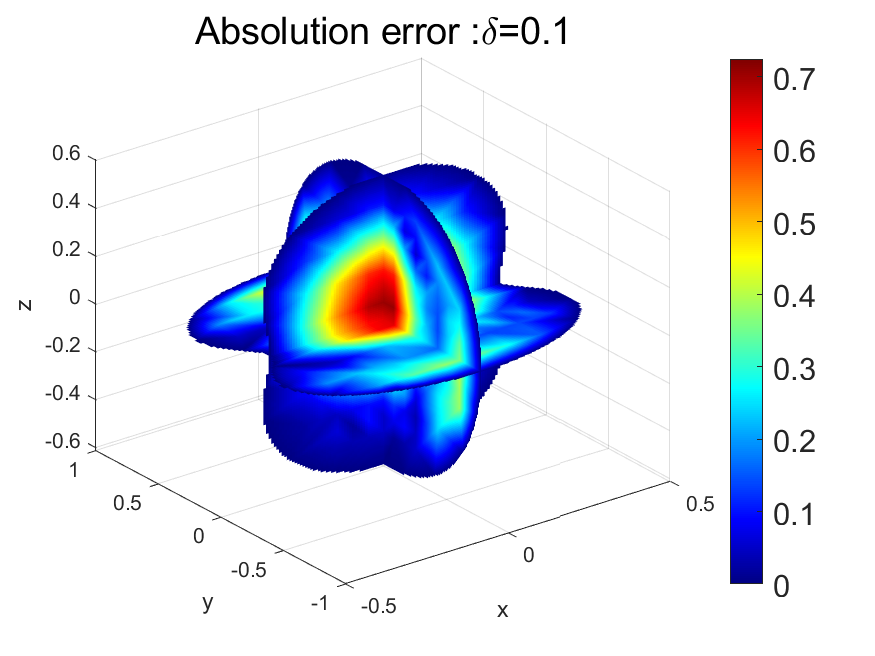}
\end{minipage}%
}%
\centering
\caption{Reconstructions of $\sigma(x,y,z)$ with different noise levels:  recovered ones (top row), absolute error distributions (bottom row).}
\label{Liu-24-recover-sigma3}
\end{figure}

To present numerical performances more clearly, which are not easy to be observed in the 3-dimensional Figure \ref{Liu-24-recover-sigma3}, we give the reconstruction result together with the exact one  in  Figure \ref{Liu-24-recover-sigma3_0} for noise level $\delta=0.05$ in three different planes. The exact conductivity and recovered one are illustrated in top row  and  bottom one in Figure \ref{Liu-24-recover-sigma3_0}. From left to right, we show the results in three different planes $x=0$, $y=0$ and $z=0$, respectively. From Figure \ref{Liu-24-recover-sigma3_0}, we can observe that the numerical performance is satisfactory. It can also be observed that the recovering results  are much deteriorated in central part of $\Omega$, which is reasonable since our inversion input data are specified in boundary domain $\Omega_{1/12}$.
\begin{figure}[htbp]
\centering
\subfigure
{
\begin{minipage}[t]{0.32\linewidth}
\centering
\includegraphics[scale=0.3]{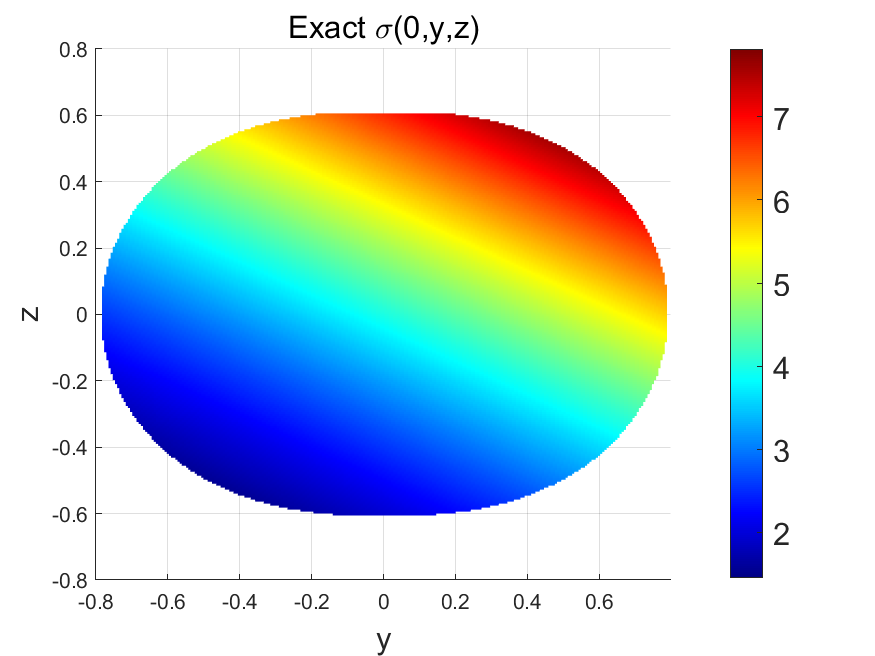}
\end{minipage}%
}%
\subfigure
{
\begin{minipage}[t]{0.3\linewidth}
\centering
\includegraphics[scale=0.3]{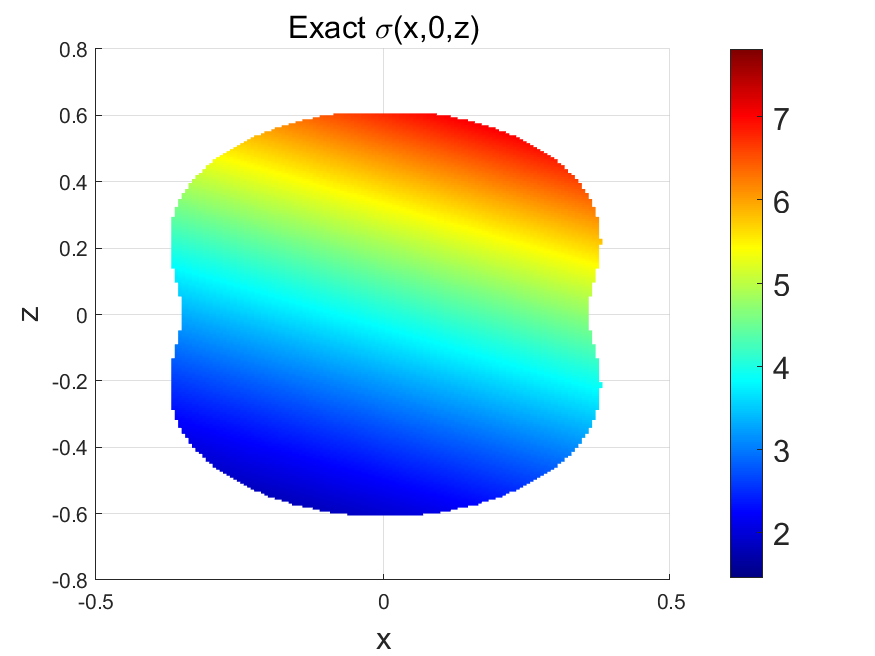}
\end{minipage}
}
\subfigure
{
\begin{minipage}[t]{0.33\linewidth}
\centering
\includegraphics[scale=0.3]{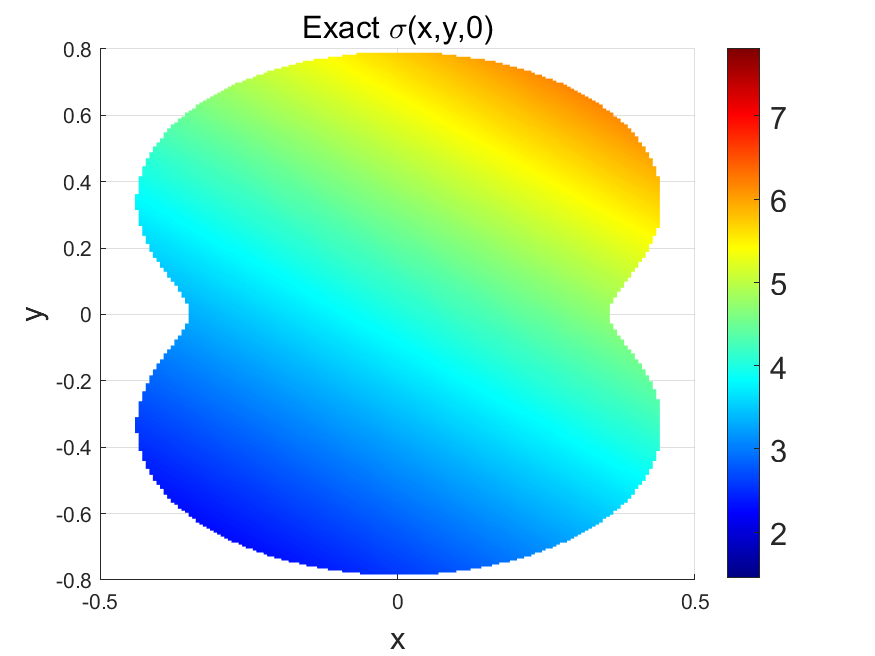}
\end{minipage}
}
\subfigure
{
\begin{minipage}[t]{0.32\linewidth}
\centering
\includegraphics[scale=0.3]{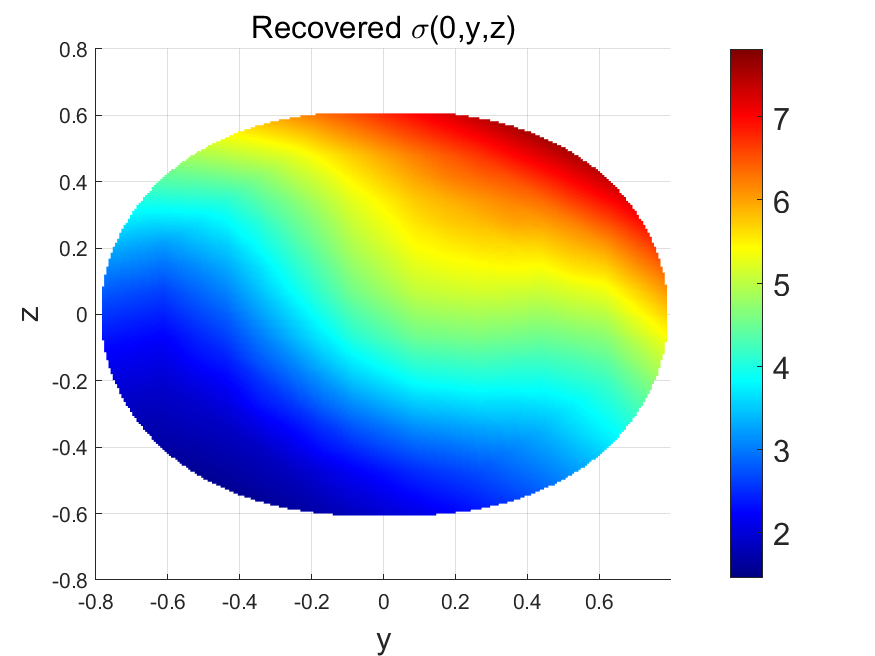}
\end{minipage}%
}%
\subfigure
{
\begin{minipage}[t]{0.3\linewidth}
\centering
\includegraphics[scale=0.3]{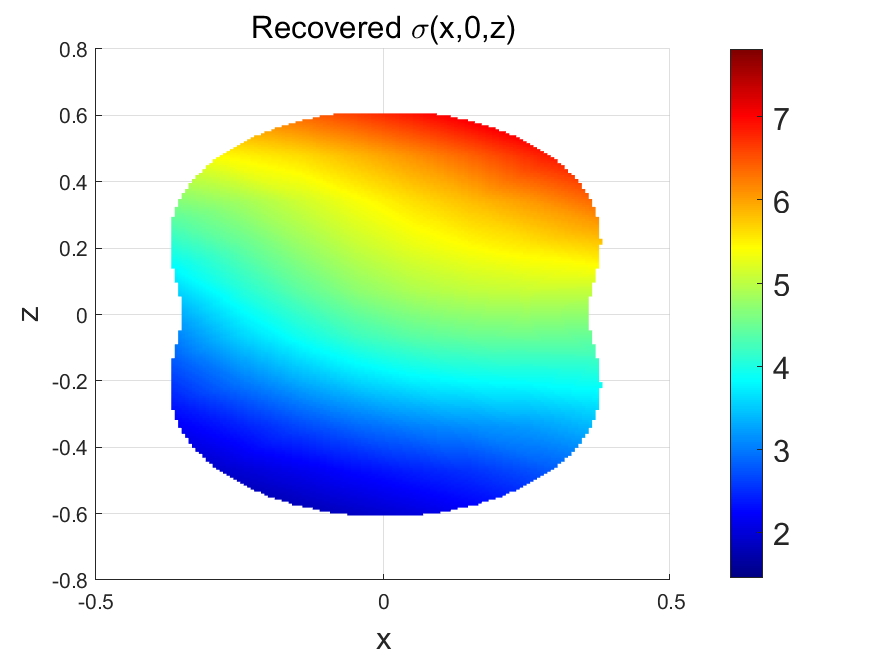}
\end{minipage}
}
\subfigure
{
\begin{minipage}[t]{0.33\linewidth}
\centering
\includegraphics[scale=0.3]{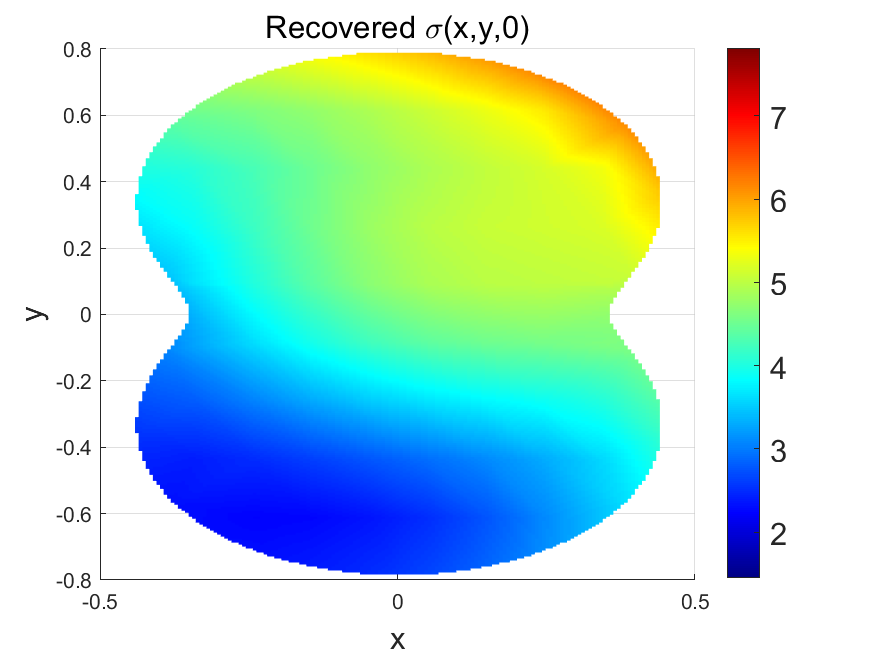}
\end{minipage}
}
\centering
\caption{Reconstructions of $\sigma(x,y,z)$ compared with exact one for noise level $\delta=0.05$ in three slice planes: $x=0$ (left), $y=0$ (middle), $z=0$ (right).}
\label{Liu-24-recover-sigma3_0}
\end{figure}

Similarly, we also give the reconstruction performances of the Robin coefficient $h(x,y,z)$ and conductivity $\sigma(x,y,z)$ quantitatively. In 3-dimensional case, the relative errors by $L^2$-norm are analogously
defined as
 \begin{align*}
\norm{\frac{h^{\delta}-h}{h}}_{L^2(\partial\Omega)}&\approx \text{RE}_{\partial\Omega}(h,\delta):=\sqrt{\frac{\pi}{N}\sum\limits_{j=1}^{N}\sum\limits_{k=0}^{2N-1}w_j\left | \frac{h^{\delta}(q_{jk})-h(q_{jk})}{h(q_{jk})} \right |^2 J(x_{jk})},\\
\norm{\frac{\sigma^{\delta}-\sigma}{\sigma}}_{L^2(\Omega)}&\approx \text{RE}_{\Omega}(\sigma,\delta):=\sqrt{\sum\limits_{l=1}^{N_t}V_l\frac{1}{4}\sum\limits_{i=1}^{4}\left |\frac{\sigma_{l(i)}^{\delta}-\sigma_{l(i)}}{\sigma_{l(i)}}\right |^2},
\end{align*}
respectively, where $w_j$ is Gauss quadrature weight and $J$ is the Jacobian of function $q$ which defines $\partial\Omega$, while $V_l$ is the volume of tetrahedron $\Delta_l$  with four vertices $l(1),\cdots,l(4)$ for $l=1,\cdots,N_t$ satisfying $\Omega\approx\bigcup _{l=1}^{N_t}\Delta_l$.
\begin{table}[H]
\centering
\caption{Reconstruction errors  of $h$ and $\sigma$ with respect to the noisy level $\delta$. }
\vspace{5pt}
\begin{tabular}{c|cccc}
\hline
Noisy level $\delta$ & 0.01 & 0.02 & 0.05 & 0.1 \\
\hline
$\text{RE}_{\partial\Omega}(h,\delta)$ & 2.244416e-3& 2.539336e-3 & 4.857161e-3 & 9.321099e-3\\

$\text{RE}_{\Omega}(\sigma,\delta)$ & 3.917395e-2 & 4.042315e-2 & 4.290789e-2 & 4.867596e-2 \\

\hline
\end{tabular}
\label{Liu-24-table-hsigma3D}
\centering
\end{table}

It can be seen from Table \ref{Liu-24-table-hsigma3D} that both $\text{RE}_{\partial\Omega}(h,\delta)$ and $ \text{RE}_{\Omega}(\Omega,\delta)$ increase as $\delta$ becomes larger. Especially for $\delta=0.05,0.1$, the error increases quickly compared to the relative small level noisy $\delta=0.01,0.02$. Moreover, it should be noticed that
$\text{RE}_{\Omega}(\sigma,\delta)$ for different noise levels are always larger than
$\text{RE}_{\partial\Omega}(h,\delta)$, since the error of recovering $h$ will propagate to the reconstruction process for conductivity $\sigma$, as analyzed in Section 3.

\section{Appendix}

We prove the estimate
\begin{equation}\label{app-1}
\norm{R_{\alpha}}_{\mathcal{L}(H^1(\Omega_\epsilon),H^2(\Omega_\epsilon))}\le C(\Omega_\epsilon,\rho,n)\frac{1}{\alpha^{n/2}}, \quad  n=2,3.
\end{equation}

Denote by $\tilde w(x)$ defined in $\mathbb{R}^n$ with $\hbox{supp}\;\tilde w\subset V$ the extension of the function $w(x)$ defined in $\Omega_\epsilon$. Then, by the same arguments in  \cite{Wangyuchan}, we have
\begin{align}\label{app-2}
    \int_{\Omega_\epsilon}\left|R_\alpha[w](x)\right|^2dx
    &\le
    \norm{\tilde w}^2_{L^2(V)}\int_{\Omega_\epsilon}dx\int_{|x-y|\le \alpha}(\rho_\alpha(|x-y|))^2dy
    \nonumber\\&\le
\norm{w}^2_{H^1(\Omega_\epsilon)}\int_{\Omega_\epsilon}dx\int_{|x-y|\le \alpha}(\rho_\alpha(|x-y|))^2dy,
\end{align}
\begin{align}\label{app-3}
    \int_{\Omega_\epsilon}\left|\nabla R_\alpha[w](x)\right|^2dx
    &\le
    \norm{\nabla \tilde w}^2_{L^2(V)}\int_{\Omega_\epsilon}dx\int_{|x-y|\le \alpha}(\rho_\alpha(|x-y|))^2dy
    \nonumber\\&\le
\norm{w}^2_{H^1(\Omega_\epsilon)}\int_{\Omega_\epsilon}dx\int_{|x-y|\le \alpha}(\rho_\alpha(|x-y|))^2dy,
\end{align}
\begin{align}\label{app-4}
    \int_{\Omega_\epsilon}\left|\Delta R_\alpha[w](x)\right|^2dx
    &\le
    \norm{\nabla \tilde w}^2_{L^2(V)}\int_{\Omega_\epsilon}dx\int_{|x-y|\le \alpha}(\rho'_\alpha(|x-y|))^2dy
    \nonumber\\&\le
\norm{w}^2_{H^1(\Omega_\epsilon)}\int_{\Omega_\epsilon}dx\int_{|x-y|\le \alpha}(\rho'_\alpha(|x-y|))^2dy.
\end{align}

So \eqref{app-2}-\eqref{app-4} leads to
\begin{align}\label{app-5}
&\hskip 0.4cm\norm{R_{\alpha}}^2_{\mathcal{L}(H^1(\Omega_\epsilon),H^2(\Omega_\epsilon))}\nonumber\\
&\le
2\left[\int_{\Omega_\epsilon}dx\int_{|x-y|\le \alpha}(\rho_\alpha(|x-y|))^2dy+\int_{\Omega_\epsilon}dx\int_{|x-y|\le \alpha}(\rho'_\alpha(|x-y|))^2dy\right].
\end{align}

For the first term in the right-hand side of \eqref{app-5},  we have
\begin{align}\label{app-6}
    \int_{\Omega_\epsilon}dx\int_{|x-y|\le \alpha}(\rho_\alpha(|x-y|))^2dy&=\int_{\Omega_\epsilon}dx\int_0^\alpha \frac{1}{\alpha^{2n}}\left(\rho(\frac{t}{\alpha})\right)^2 2^{n-1}\pi t^{n-1}dt
\nonumber\\
&=\frac{1}{\alpha^n}|\Omega_\epsilon|\int_0^1(\rho(t))^2
2^{n-1}\pi t^{n-1}dt:=
C(\rho,n,\Omega_\epsilon)\frac{1}{\alpha^n}.
\end{align}
For the second term, we also have
\begin{equation}\label{app-7}
    \int_{\Omega_\epsilon}dx\int_{|x-y|\le \alpha}(\rho_\alpha(|x-y|))^2dy=
\frac{1}{\alpha^n}|\Omega_\epsilon|\int_0^1(\rho'(t))^2
2^{n-1}\pi t^{n-1}dt:=C(\rho',n,\Omega_\epsilon)\frac{1}{\alpha^n}.
\end{equation}

Combining \eqref{app-5}-\eqref{app-7} together, we obtain \eqref{app-1}.

\begin{rem}
In \cite{Wangyuchan}, we established a weak estimate
$\norm{R_{\alpha}}\le C(\Omega_\epsilon,\rho,2)\frac{1}{\alpha^{2}}$ in the case $n=2$.
\end{rem}

\vskip 0.3cm

{\bf Acknowledgements: } This work is supported by National Key R$\&$D Program of China (No.2020YFA0713800), and Jiangsu Provincial Scientific Research Center of Applied Mathematics under Grant No. BK20233002.

\end{document}